\documentclass[11pt]{amsart}
\usepackage[pagebackref,colorlinks=true,pdfpagemode=none,urlcolor=blue,linkcolor=blue,citecolor=blue]{hyperref}

\usepackage{
amssymb, graphicx, mathrsfs,xcolor, sidecap 
}
\usepackage[normalem]{ulem}

\graphicspath{{./figures/}}

\numberwithin{equation}{section}%

\newtheorem{theorem}{Theorem}[section]

\newtheorem{corollary}{Corollary}[section]

\theoremstyle{definition}
\newtheorem{remark}{Remark}[section]

\newtheorem{example}{Example}[section]

\DeclareMathOperator{\supp}{supp}
\DeclareMathOperator{\Ker}{Ker}

\DeclareMathOperator{\spec}{spec}
\DeclareMathOperator{\WF}{WF}

\newcommand{\eps}{\varepsilon}

\newcommand{\Id}{\mbox{Id}}
\renewcommand{\r}[1]{(\ref{#1})}
\newcommand{\PDO}{$\Psi$DO}
\newcommand{\be}[1]{\begin{equation}\label{#1}}
\newcommand{\ee}{\end{equation}}

\def \Rm {\mathbb{R}}

\newcommand{\cD}{\mathscr{D}}
\newcommand{\cE}{\mathscr{E}}

\renewcommand{\d}{\mathrm{d}}

\newcommand{\bo}{\partial \Omega}
\renewcommand{\L}{\mathcal{L}}

\title[The attenuated geodesic X-ray transform]{The attenuated geodesic X-ray transform}

\author[S. Holman]{Sean Holman}
\address{School of Mathematics, University of Manchester, Manchester, UK M13 9PL}
\author[F. Monard]{Fran\c{c}ois Monard}
\address{Department of Mathematics, University of California, Santa Cruz, CA 95064}
\author[P. Stefanov]{Plamen Stefanov}
\address{Department of Mathematics, Purdue University, West Lafayette, IN 47907}
\thanks{P.S.\ partly supported by  NSF  Grant DMS--1600327, F.M.\ partly supported by NSF grant DMS--1712790, S.H.\ partly supported by the Engineering and Physical Sciences Research Council (EP/M016773/1)}
\date{\today}

\begin{document} 

\begin{abstract} This article deals with stability issues related to geodesic X-ray transforms, where an interplay between the (attenuation type) weight in the transform and the underlying geometry strongly impact whether the problem is stable or unstable. In the unstable case, we also explain what types of artifacts are expected in terms of the underlying conjugate points and the microlocal weights at those points. We show in particular that the well-known iterative reconstruction Landweber algorithm  cannot provide accurate reconstruction when the problem is unstable, though the artifacts generated, specific for the reconstruction algorithm, can be properly described. 
\end{abstract}

\maketitle 

\section{Introduction} 

Continuing prior work on the analysis of X-ray transforms with conjugate points \cite{SU-caustics,MonardSU14,Holman2015}, we provide a thorough analysis of the local and global stability of attenuated X-ray transforms on non-trapping surfaces, discussing the impact on stability of the interplay between conjugate points and the microlocal weights in the transform. Given $(M,g)$ a non-trapping Riemannian manifold with strictly convex boundary and $0\le a\in C^\infty(M,\Rm)$, the {\em attenuated geodesic X-ray transform} is the mapping $X_a:C_c^\infty(M^\text{int})\to C_c^\infty(\partial_+ SM)$ defined by 
\begin{align}
    X_a f(x,v) = \int_0^{\tau(x,v)} f(\gamma_{x,v}(t)) e^{-\int_0^{\tau(x,v)} a(\gamma_{x,v}(s),\dot\gamma_{x,v}(s))\ ds}\ dt, \qquad (x,v)\in \partial_+ SM,
    \label{eq:Xa}
\end{align}
extendible by duality as $X_a:\cE'(M^\text{int})\to \cE'(\partial_+ SM)$\footnote{because the operator $X_a^*:C^\infty(\partial_+ SM)\to C^\infty(M)$ is continuous.}, and where in the equation above, $\tau(x,v)$ denotes the first exit time of the geodesic starting at $(x,v)$ and $\partial_+ SM$ is the inward bundle 
\begin{align*}
    \partial_+ SM = \{ (x,v) \in \partial(TM),\ |v| = 1, \quad g(v,\nu_x) >0  \},
\end{align*}
and $\nu_x$ is the inner normal at $x\in \partial M$. Here and below, norms of vectors and covectors are taken w.r.t.\ the metric. 
We also consider partial data cases in which $X_a$ is defined by the same formula but only known for $(x,v)$ in some open subset of $\partial_+ SM$. 

Such a transform, generalizing the extensively studied unattenuated case (see, e.g., \cite{Sh-book, S-Serdica, SU-AJM,SU-Duke,SU-JAMS,SU-lens,UV:local,SUV-tensors, MonardSU14, Holman2015} and the references there), is a model for X-ray tomography in media with variable refractive index \cite{Manjappa2015,Nguyen2014}.  The Euclidean version of the transform has also been extensively studied for its applications to Single Photon Emission Computerized Tomography, see \cite{Finch_03} for a topical review, and the inversion techniques were generalized to the hyperbolic case in \cite{Bal2005}. Past these constant curvature cases, the next ``best'' case where attenuated X-ray transforms are understood to be injective and stable\footnote{more specifically, mildly ill-posed or order $1/2$} is when $(M,g)$ is {\em simple}, that is, when $\partial M$ is strictly convex and $M$ has no conjugate points in its interior. In this case, injectivity and stability were proved in \cite{SaloU11} and inversions were given in \cite{Monard2015}. Such a transform can also be considered over vector fields (the so-called {\em Doppler transform} \cite{Kazantsev2007,S-Sean-Doppler,Sadiq2014}), or higher-order tensor fields \cite{Sadiq2015,Monard2017}. Once this simplicity condition is violated by the presence of conjugate points, stability is at stake and involves the interplay of a few factors, as explained below. 

Instability here is described in terms of the presence of a non-empty {\em microlocal kernel}. By ``microlocal kernel'',   $\mu\ker X$,  of an operator $X$ here we mean the space of distributions, modulo smooth functions, whose images by $X$ are smooth functions. While the presence of a microlocal kernel says nothing about injectivity of the operator (e.g., if $A:\cE'(\Rm)\to \cD'(\Rm)$ denotes convolution by a Gaussian, $A$ is injective yet its microlocal kernel is all of its domain), its non-emptiness implies that inverting $A$ is, globally, an unstable (or severely ill-posed) problem. Namely, the presence of a microlocal kernel prevents the possibility of any global stability estimate of the form
\begin{align*}
    \|f\|_{H^{s_1}} \le C\left(\|Xf\|_{H^{s_2}} + \|f\|_{H^{s_3}}\right),       
\end{align*}
no matter the choice of Sobolev indices $s_1,s_2,s_3$, as was previously observed in  \cite{SU-JFA09} and in \cite{MonardSU14} for this particular problem. Note that for some of the cases analyzed below, stable or unstable, proving injectivity is still open (yet conjectured to be true). In two dimensions, if the metric is non-trapping and if the metric and the weight are analytic, then there is injectivity as follows from  the analytic microlocal arguments used in \cite{SU-AJM}. 

Even if the problem is globally severely unstable, microlocal analysis allows refinement of the notion of stability, which is one of the goals of this work --- to explain what kind of artifacts are possible or unavoidable. The notion of invisible singularities in integral geometry and other inverse problems refers to an open conic set $\Gamma$ so that \textit{for every $f$} with $\WF(f)\subset\Gamma$, the data (in this case, $X_af$) is smooth. One could call the other singularities visible but that does not really mean that they can be recovered stably: first, singularities on the boundary of $\Gamma$ are a borderline case in terms of stability, but most importantly, singularities in the complement of $\bar\Gamma$ can cancel each other. An example of this phenomenon is present in SAR imaging modeled by integrals of a function in the plane over circles centered at a fixed line (the flight path). Singularities symmetric about that line can cancel each other and give even zero measurements; this is known as the left-right ambiguity in SAR, see, e.g., \cite{SU-SAR} for references and even more general results. On the other hand, such singularities are not invisible because \textit{for some} $f$ with $\WF(f)$ in the corresponding $\Gamma$, the measurement is not smooth. 

In the problem we study here, inversion of $X_a$ possibly with partial data, invisible singularities would be $(x,\xi)$ so that there is no geodesic in our family (assumed to be open) through $x$ conormal to $\xi$. If $(M,g)$ is non-trapping and we have full data, there are no invisible singularities. On the other hand, we can and do have a non-trivial microlocal kernel if $n=2$ and the weight is constant consisting of suitable distributions having wave front sets at pairs of conjugate points over some geodesic and conormals to it at those points. Singularities at such pairs cannot be recovered stably, and on a microlocal level we have one equation for two unknowns. If the attenuation is non-trivial, the direction of the integration matters and we get two linearly independent equations; then the singularities can be recovered.

Formulated in terms of the microlocal kernel, the stability results proved in \cite{MonardSU14} concerning the attenuated transform $X_a:\cE'(M^\text{int})\to \cE'(\partial_+ SM)$ defined in \eqref{eq:Xa}, in dimension $n = 2$, are summarized as follows:
\begin{itemize}
    \item If $a=0$ and there exist two conjugate points along some geodesic in $M$, then $\mu\ker X_a \ne \{0\}$
 and the problem is globally unstable. 
    \item If there exists three or more conjugate points along some geodesic in $M$, then $\mu\ker X_a \ne \{0\}$ and the problem is globally unstable no matter the choice of $a$. 
    \item If $a>0$ (or $a<0$)  and if no more than two conjugate points exist along any given geodesic in $M$, then $\mu\ker X_a = \{0\}$ and the problem is globally stable. In fact, it is enough to have a non-zero attenuation between each pair of conjugate points only.    
\end{itemize}
In fact, we study the more generalized weighted X-ray transform. 
As such statements are local in nature, one may also reason in the neighborhood of a fixed geodesic, as will be done below. 

In dimensions $n\ge3$, we have the following, see also Section~\ref{sec_2.7}:
\begin{itemize}
    \item[$(i)$] Stability might hold in three dimensions and higher in the presence of conjugate points (and $a=0$) under additional assumptions. In this case, following ideas first developed in \cite{SU-AJM, UV:local}, stability can be proved for ray transforms associated with a wide range of curves and weights under a foliation condition (a condition which allows conjugate points) \cite{Zhou2013,Paternain2016}. We illustrate here with numerical examples how, in some unstable two-dimensional scenarios, the three-dimensional counterpart becomes stable. The reason is that a single singularity can be potentially resolved by geodesics conormal to it (and a small neighborhood of such) forming an $n-2$ dimensional submanifold. If some of them have conjugate points, others may not, and we can use the latter to resolve the singularity. Once we have resolved some singularities, we may use that in a layer stripping argument to resolve even more, as done in \cite{UV:local}.  This argument makes a possible recovery a non-local problem. 
    \item[$(ii)$] In higher dimensions, it is unclear how to make statements concerning a single geodesic similar to the ones above for the two dimensional case, mainly because conjugate points are not all of order one, and the normal operator $X_a^* X_a$ has Fourier Integral components of possibly non-graph type whose Sobolev mapping properties depend on the order of conjugate points present, some of which remain relatively compact with respect to the pseudo-differential part of the normal operator, see \cite{Holman2015}. Generalizing such statements will be the object of future work. 
\end{itemize}

Finally, in the presence of a non-trivial microlocal kernel, we explain what the Landweber iterative reconstruction scheme converges to. Namely, we show that such a method, initially designed to solve the least norm solution to the problem $X_a^* X_a f = X_a^* g$, converges to a solution which produces artifacts split equally among the conjugate points, and of course fails to reconstruct any part of $f$ which is in the microlocal kernel. This is illustrated with various numerical examples confirming the theoretical predictions. 

{\bf Outline.} The rest of the article is organized as follows. Section \ref{sec:theory} covers microlocal results describing the stability or instability of the attenuated geodesic X-ray transform, and the location and strength of the artifacts obtained in the unstable case. Section \ref{sec3} describes outcomes of the Landweber iteration, in particular its failure to completely reconstruct some aspects of the unknown $f$ in the unstable case. Section \ref{sec:numerics} contains numerical ilustrations of the claims made in the previous sections.


\section{Theory} \label{sec:theory}
\subsection{Microlocal preliminaries} \label{sec:prelim} 
In this subsection we introduce some of the microlocal concepts we will use in the ensuing analysis of $X = X_0$ and $X_a$. For any conic open set $\Gamma\subset T^* M\backslash 0$, we define the microlocal space $L^2(\Gamma)$ as the space of distributions $f \in \cE'(M)$ for which $Pf\in L^2(M)$ for any properly supported zeroth order pseudo-differential operator (\PDO) $P$ with microsupport in $\Gamma$, see, e.g., \cite{Treves}. One can use $\|Pf\|_{L^2(M)}$ as a family of seminorms. 

Let the \PDO\ $\Lambda  = \sqrt{-\Delta_{g}}$ modulo smoothing operators be properly supported and in what follows, powers of $\Lambda$ are considered modulo smoothing operators as well.  
Then we use $\Lambda^{-s}$ as an isomorphism, by definition, between $L^2(\Gamma)$ and  the {\it microlocal Sobolev spaces} $H^s(\Gamma)$, i.e., $f\in H^s(\Gamma)$ if and only if $\Lambda^s f\in L^2(\Gamma)$. We also declare this isomorphism to be unitary modulo lower order \PDO s; which is true globally in a classical sense if we define Sobolev spaces by the Fourier transform with weights $(|\xi|^2_g+1)^{s/2}$. This allows us to talk about principally unitary operators between microlocal Sobolev spaces, once we define microlocal unitarity in $L^2$ below. 

Now let us consider operators between manifolds, and so let $(M_1,g_1)$ and $(M_2,g_2)$ be two Riemannian manifolds, and $\Gamma_1 \subset T^* M_1\backslash 0$ and $\Gamma_2 \subset T^* M_2\backslash 0$ be open conic sets. Suppose that $U$ is a linear operator, initially defined from $C_c^\infty(M_1)$ to $C_c^\infty(M_2)$ and assumed to be extendible by duality to a map from $\cE'(M_1)$ to $\cE'(M_2)$. If not otherwise specified the adjoint $U^*$ will be defined using the $L^2$ inner products on $(M_1,g_1)$ and $(M_2,g_2)$, although we will also use the adjoint defined between Sobolev spaces discussed in the previous paragraph. Indeed, in this way the adjoint $U^*$ of $U$ acting from $H^{s_1}$ to $H^{s_2}$ will be defined by the requirement
\[
\langle \Lambda^{s_2} Uf, \Lambda^{s_2} g \rangle_{L^2(M_2)} = \langle \Lambda^{s_1} f, \Lambda^{s_1} U^* g \rangle_{L^2(M_1)}
\]
for all $f \in C_c^\infty(M_1)$ and $g \in C_c^\infty(M_2)$. This implies that the $L^2$ adjoint of $U$, which we here call $U^*_{L^2}$, and the adjoint $U^*$ of $U$ acting from $H^{s_1}$ to $H^{s_2}$ are related by
\[
U^* = \Lambda^{-2s_1} U^*_{L^2} \Lambda^{2 s_2}.
\]
We will generally be working in this context with Fourier integral operators (FIOs) associated to canonical graphs. The {\it microlocal kernel} $\mu \ker U$, will be defined to be the set of equivalence classes modulo $C^\infty_c(M_1)$ of $f \in \cE'(M_1)$ such that $Uf \in C^\infty(M_2)$. Note that it is always the case that $\{0\} \in \mu \ker U$. 
Next we consider microlocal notions of unitarity. For this let $\Gamma_1 \subset T^* M_1\backslash 0$ and $\Gamma_2 \subset T^* M_2\backslash 0$ be open conic sets and suppose that upon restricting $U$ to $L^2(\Gamma_1)$ we have the property $U:L^2(\Gamma_1)\to L^2(\Gamma_2)$. We then say that $U$ is {\it microlocally unitary} if $U^*U-\Id$ is smoothing in $\Gamma_1$ and $UU^*-\Id$  is smoothing in $\Gamma_2$. If $U$ is an elliptic FIO with a diffeomorphic canonical relation mapping $\Gamma_1$ to $\Gamma_2$, then one of those implies the other because then we can apply a parametrix. If instead the restriction maps as $U: H^{s_1}(\Gamma_1)\to H^{s_2}(\Gamma_2)$, then we say $U$ is {\em microlocally unitary} on these spaces if $\Lambda^{s_2}  U\Lambda^{-s_1} :  L^2(\Gamma_1)\to L^2(\Gamma_2)$ is microlocally unitary. We say that those operators are {\em principally unitary} if the smoothing errors above are replaced by \PDO s of order $-1$.  Note that $U:H^{s_1}(\Gamma_1)\to H^{s_2}(\Gamma_2)$ is microlocally (resp., principally) unitary if and only if for the adjoint $U^*$ of $U$ from $H^{s_1}$ to $H^{s_2}$, as defined in the previous paragraph, $U U^* - \mathrm{Id}$ and $U^* U - \mathrm{Id}$ are smoothing operators (resp., \PDO s of order $-1$) in $\Gamma_1$ and $\Gamma_2$ respectively. 

Finally, for $f_k\in H^{s_k}(\Gamma_k)$ with $\WF(f_k)\subset \Gamma_k$, $k=1,2$, we say that $f_1$ and $f_2$ {\em have the same strength} (in the corresponding microlocal Sobolev spaces) if there is a principally unitary $P: H^{s_1}(\Gamma_1)\to H^{s_2}(\Gamma_2)$ so that $f_2=Pf_1$.  

\subsection{Microlocal analysis of $X$ near a single directed geodesic} \label{sec:singleGeodesic} 
Let $(M,g)$ be a Riemannian manifold, say complete, for convenience. We will actually study the weighted geodesic ray transform, which is more general than \eqref{eq:Xa}, given by
\be{01}
Xf(\gamma) =\int \kappa(\gamma(s),\dot\gamma(s)) f(\gamma(s))\, \d s
\ee
for $\gamma$ in an open set of (directed) unit speed geodesics where $\kappa$ is a smooth non-vanishing weight, homogeneous in its second variable of degree zero. We always assume that those geodesics intersect $\supp f$ in a compact set, and are non-trapping for that set; i.e., they leave it in both directions. Note that the attenuated X-ray transform \r{eq:Xa} is a weighted transform with weight
\be{P1}
\kappa (x,v) = e^{-\int_0^{\infty}a(\gamma_{x,v}(s), \dot \gamma_{x,v}(s)) \, \d s},
\ee
where $a(x,v)\ge0$ is the attenuation; in this case $0<\kappa\le 1$.  
Also, the weight \eqref{P1} increases along the geodesic flow. More precisely, if $G$ is the generator of the geodesic flow, we have
\[
G\log \kappa = a\ge0.
\]
If $a>0$, then $\kappa$ is strictly increasing.

We localize the problem first, near $\gamma_0$ a fixed geodesic. We take a finite segment of it, and call it $\gamma_0$ again. In this section we study $Xf(\gamma)$ for $\gamma$ belonging to a small neighborhood  $\mathcal{M}$ of $\gamma_0$. Below in Section \ref{sec:bothdirections} we also consider $Xf(\gamma)$ for $\gamma$ close to $\gamma_0(-\cdot)$, i.e., with the direction reversed; which gives us different information if $\kappa$ is not an even function of $v$. We always assume that $\supp f\subset K$ with some compact set $K$  disjoint from the endpoint of $\gamma_0$ and from the endpoints of all $\gamma$'s in $\mathcal{M}$. 

We parameterize $\mathcal{M}$ by taking a hypersurface (a curve in 2D)  $H$ transversal to $\gamma_0$ and using the intersection $p$ with $H$ and the projection $\theta'$ of the direction $\theta= \dot\gamma$ to $T_pH$. There is a natural measure on  $TH$, and by Liouville's theorem, that measure is invariant under a different choice of a transversal $H$. We define $L^2(\mathcal{M})$ w.r.t.\ that measure.  Covectors in $T^*H$ can be naturally identified with Jacobi fields normal to the trivial ones: $\dot \gamma(t)$ and $t\dot\gamma(t)$. We model $\gamma\in\mathcal{M}$ by choosing $\supp\kappa$ appropriately. 

If there are no conjugate points, to construct a microlocal parametrix to $X$, we first set $N=X^*X$ to be the normal operator, where the $L^2$ adjoint $X^*$ is taken with respect to the natural measure on $\mathcal{M}$. Explicitly, if $\d_x\gamma$ is a natural restriction of that measure to $\{\gamma, \gamma\ni x\}$, we have
\[
X^*\psi(x) = \int_{\gamma\ni x} \bar\kappa \psi(\gamma) \,\d_x\gamma.
\]
Then $N = X^* X$ is a \PDO\ of order $-1$ with principal symbol 
\be{pr_s}
a_{-1}(x,\xi) = 2\pi \int_{S_xM}|\kappa(x,\theta)|^2 \delta(\langle\xi,\theta\rangle) \, \d\sigma_x(\theta),
\ee
where $\d\sigma_x(\theta)$ is the natural volume measure on $S_xM$, see, e.g., \cite{SU-Duke}.  If $n=2$, then the integral above is a sum of two terms:
\be{pr_s2}
a_{-1}(x,\xi) = \frac{2\pi}{|\xi|} \left(|\kappa(x,\xi_\perp)|^2 + |\kappa(x,-\xi_\perp)|^2\right).
\ee
Here $\xi_\perp = (\det g)^{-1/2}(-\xi_2,\xi_1)$ is the vector conormal to the covector $\xi$, of the same length, and rotated by $\pi/2$ in the fixed coordinate system which defines an orientation near $x$. We fix that orientation and define a positive side of $\gamma_0$ corresponding to normal vectors which can be obtained from $\dot\gamma$ by a rotation by $\pi/2$ in that system. Given a vector $v\in T_xM$, $v\to v^\perp= (\det g)^{1/2}(v^2,-v^1)\in T^*_xM$ is the inverse operator.

Formula \eqref{pr_s} shows that $N$ is elliptic at $(x,\xi)$ if and only if there is $\theta\in S_xM$ so that $\langle \xi, \theta\rangle = 0$ and $\kappa(x,\theta)\not=0$. In the two dimensional case, which we will consider for the rest of this section, $\kappa$ vanishes near one of the directions, say $-\xi_\perp$, because of the localization of $\gamma$ to $\mathcal{M}$; therefore only the first term of \eqref{pr_s2} remains. In that case, $N$ has a parametrix $N^{-1}$ and $N^{-1}X^*$ is a parametrix for $X$.  We refer to \cite{SU-Duke,SU-JAMS,FSU} for more details. 
Assume now that there are pairs of conjugate points along $\gamma_0$, in $K$. As in \cite{MonardSU14},  let $(p_1,p_2)$ be such a pair on $\gamma_0$, and let $v_1$ and $v_2$ be the unit speeds at $p_1$, $p_2$, respectively, see Figure~\ref{pic01}. Assume for simplicity that there are no other points in $\gamma_0$ conjugate to $p_1$ (and $p_2$). 
 Fix small neighborhoods $U_1$ and $U_2$ of $p_1$ and $p_2$, respectively. Then $Xf$, restricted to $\mathcal{M}$ can only possibly ``detect'' singularities in $U_1\cup U_2$ in a neighborhood of those conormal to $\gamma_0$; we denote this conic subset of $T^*M\setminus 0$ of ``visible singularities" by $V$. Note that those conormals near each point have two possible directions, and this naturally splits $V$ into two disconnected components $V=V_-\cup V_+$, where $V_+$ are the covectors consistent with a fixed orientation of $\gamma_0$, and $V_-$ are the rest. Let $(p_j,\xi^j)\in V_+$ be conormal to $\gamma_0$. 
  We are interested in recovery of singularities in conic neighborhoods of those points; and for that reason, we take $V_\pm^j\subset V_\pm$ to be small conic neighborhoods of $(p_j,\pm\xi^j)$, $j=1,2$.

\begin{figure}[!ht] 
  \centering
  \includegraphics[width = 3in,page=2]{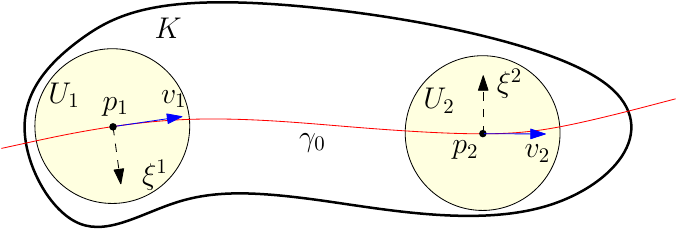}
  \caption{Setup: $p_1$ and $p_2$ are conjugate along $\gamma_0$. Singularities of $f$ at $(p_1,\xi^1)$ and  $(p_2,\xi^2)$, affect $\WF(Xf)$ at the same point and may not be resolvable. 
  }
  \label{pic01}
\end{figure}

In \cite{MonardSU14}, we showed that the operator $X: C_0^\infty(K)\to C^\infty(\mathcal{M})$ is an FIO of order $-n/4$ associated with a canonical relation  $\mathcal{C}$ described in more detail  there. 
When $n=2$, $X$ is of order $-1/2$ and the domain and the range of $\mathcal{C}$ are of the same dimension, $2$. The canonical map $\mathcal{C}$ then is a local diffeomorphism which is global  if and only if there are no conjugate points on $\gamma_0$; and if there are no such points, $X$ is elliptic if $\kappa$ does not vanish, which we assume. When there are conjugate points, as we assume here, $\mathcal{C}(V^1)$ and $\mathcal{C}(V^2)$ are mapped to the same conic neighborhood of $\mathcal{C}(p_1,\xi^1)$ and $\mathcal{C}(p_2,\xi^2)$ which coincide up to a multiplication by non-negative  factor. Without loss of generality we assume $\mathcal{C}(V^1)=\mathcal{C}(V^2) $ and we call the latter set $\mathcal{V}$. We showed in \cite{MonardSU14} that $\mathcal{C}(V^1_\pm)=\mathcal{C}(V^2_{\mp}) $.   For this reason, we change the direction of $\xi^2$ if needed to make sure that $(x_1, \xi^1)\in V_+^1$ and $(x_2, \xi^2)\in V^2_-$. Then  
\be{vxi}
\xi^1/|\xi^1| = v_1^\perp , \qquad  -\xi^2/|\xi^2| = v_2^\perp,
\ee
see Figure~\ref{pic01}. The problem then is reduced to the following: which singularities of $f$ in ${V}_\pm$ can we reconstruct from knowing $\WF(Xf)$ in $\mathcal{V}$? Note that $\WF(Xf)$ outside $\mathcal{V}$ cannot determine anything of $\WF(f)$ in $V$. 

One of the main results in \cite{MonardSU14} is that $Xf$ known near $\gamma_0(t)$ (but not near $\gamma_0(-t)$) does not recover $\WF(f)$ near $(p_1,\xi^1)$ and  $(p_2,\xi^2)$. We present the arguments below. Let $X_k$ be $X$ micro-localized to functions with wavefronts in $V^k$, $k=1,2$, and restricted to $\gamma\in \mathcal{M}$. More precisely, for $k=1,2$, we take $\chi_k$ a zeroth order \PDO\ with essential support in some conic neighborhood of $V^k$ and full symbol (in any local chart) equal to $1$ in $V^k$. Then we set $X_k=X\chi_k$, and we will study those operators for $\gamma\in \mathcal{M}$ only. Then $X_k$ is an elliptic FIO with canonical relation associated to the canonical diffeomorphism $\mathcal{C}_k = \mathcal{C}|_{V^k}$, and since each neighborhood of $p_{1,2}$ can be chosen small enough to not contain conjugate pairs, the operators $X_1$, $X_2$ admit parametrices (call them $X_1^{-1}$ and $X_2^{-1}$) as explained earlier.  

Take $f=f_1+f_2$ with $f_k$ singular in $V^k$ only, $k=1,2$. Then we write
\[
Xf= X_1f_1 + X_2f_2,
\]
with all equalities here and below understood as equalities modulo smooth terms. Since $X_k$ are elliptic from $V^k$ to $\mathcal{V}$, we get
\be{cancel}
Xf\in H^s(\mathcal{V})    \Longleftrightarrow X_2^{-1}X_1f_1 +f_2 \in H^{s-1/2}(V^2) 
 \Longleftrightarrow  f_1 +X_1^{-1}X_2f_2 \in H^{s-1/2}(V^1).
\ee
One of the implications of \r{cancel} is that certain singularities are unrecoverable. Given any $f_1$ with $\WF(f_1)\subset V^1$, we can find $f_2$ as above, so that $f=f_1+f_2$ does not create singularities of $Xf$, and we can switch $f_1$ and $f_2$ in that statement. Moreover, we have a description of the microlocal kernel. Note that $X_2^{-1}X_1$ and $X_1^{-1}X_2$ are FIOs with canonical relations associated to canonical diffeomorphisms $\mathcal{C}_{21}= \mathcal{C}_2^{-1} \circ \mathcal{C}_1 : V^1\to V^2$ and $\mathcal{C}_{12}= \mathcal{C}_1^{-1} \circ \mathcal{C}_2 : V^2\to V^1$.  As shown in \cite{MonardSU14}, $\mathcal{C}_{12}$ is the twisted version (the second dual variable changes sign) of the conormal bundle $N^*Z$ of the conjugate locus $Z$ of pairs $(p,q)$ conjugate to each other along a geodesic close to $\gamma_0$, and clearly, $\mathcal{C}_{21}= \mathcal{C}_{12}^{-1}$. With the definitions from Section \ref{sec:prelim}, the theorem below refines the properties of the operators $X_1^{-1}X_2$ and $X_2^{-1} X_1$. 

\begin{theorem}\label{thm1}
Let $n=2$.
\begin{enumerate}
    \item[(a)] If $\kappa=1$, then 
	\[
	    \begin{split}
		F_{21}&:= X_2^{-1}X_1: H^{-1/2}(V^1)\to H^{-1/2}(V^2),\\
		F_{12}&:= X_1^{-1}X_2: H^{-1/2}(V^2)\to H^{-1/2}(V^1)
	    \end{split}
	\]
	are principally unitary.  In particular, if $F^\sharp$ denotes the adjoint as a map on $H^{-1/2}$, then
	\be{F*eq}
	F_{21}^\sharp = F_{12}, \ \mbox{and}\ F_{12}^\sharp = F_{21}
	\ee
	in $V^2$ and $V^1$ respectively modulo FIOs with canonical relations $\mathcal{C}_{12}$ and $\mathcal{C}_{21}$ of order $-1$.
    \item[(b)] For general $\kappa\not=0$, $|\kappa(x,-D_\perp)| F_{21} |\kappa^{-1}(x,D_\perp)|$ and $|\kappa(x,D_\perp)| F_{12}| \kappa^{-1}(x,-D_\perp)|$ are  principally unitary in the spaces above.
\end{enumerate}
\end{theorem}
\begin{proof}Let us write $N_j = X_j^* X_j$ for $j = 1$ or $2$. Then the operators 
\[
X_1 N_1^{-1/2}:L^2(V^1)\longrightarrow  L^2(\mathcal{V}), \quad 
X_2 N_2^{-1/2}:L^2(V^2)\longrightarrow  L^2(\mathcal{V})
\]
are principally unitary because 
\[
\left(X_1N_1^{-1/2}\right)^* \left(X_1N_1^{-1/2}\right)= N_1^{-1/2}X_1^*X_1N_1^{-1/2} = N_1^{-1}X_1^*X_1 = \Id\qquad \text{mod $\Psi^{-1}$}. 
\] 
 Then the same is true for each of the two operators in the parentheses below 
\be{XN}
F_{21} = X_2^{-1}X_1 =  N_2^{-1/2}\left(X_2 N_2^{-1/2}\right)^{-1}   \left(X_1 N_1^{-1/2}\right) N_1^{1/2}.
\ee
Since $N_k= c\Lambda^{-1}$ modulo $\Psi^{-2}$ (in $V^k$), $k=1,2$, this proves that $\Lambda^{-1/2} F_{21} \Lambda^{1/2}: L^2(V^1) \rightarrow L^2(V^2)$ is principally unitary, and so completes the proof of principal unitarity in (a) for $F_{21}$, and the proof for $F_{12}$ is the same. For \eqref{F*eq} note that the principal unitarity says precisely that
\[
F_{21} F_{21}^\sharp=  \Id\qquad \text{mod $\Psi^{-1}$.}
\]
Since $F_{12}$ is a local parametrix for $F_{21}$ as well, \eqref{F*eq} follows.

To prove (b), notice that for general non-vanishing $\kappa$, $N_1 = c|\kappa(x, D_\perp)|^2\Lambda^{-1}$ in $V^1$,  and $N_2 = c|\kappa(x, -D_\perp)|^2\Lambda^{-1}$ in $V^2$  mod $\Psi^{-2}$. This, combined with \r{XN} proves (b) for $F_{21}$. The proof for $F_{12}$ is similar. 
\end{proof}

\begin{remark}\label{remark_1.1}
One may wonder if in Theorem \ref{thm1}(a), we actually have $L^2$, rather than $H^{-1/2}$, unitarity. To address this question we can apply Egorov's theorem to $\Lambda^{-1/2} F_{21} \Lambda^{1/2}$, which from the proof of Theorem \ref{thm1} we know is $L^2$ principally unitary, to commute $\Lambda^{1/2}$ past $F_{21}$. We will then have, modulo lower order operators,
\be{LFL}
\Lambda^{-1/2} F_{21} \Lambda^{1/2} = \Lambda^{-1/2}\ \text{Op}\left (\mathcal{C}_{12}^* \sigma_p(\Lambda^{1/2})\right ) F_{21}.
\ee
Here $\text{Op}(\mathcal{C}_{12}^* \sigma_p(\Lambda^{1/2}))$  is a $\Psi$DO with principal symbol given by pulling back the principal symbol of $\Lambda^{1/2}$ by $\mathcal{C}_{12}$. Thus $F_{21}$ is principally unitary on $L^2$ if and only if $\Lambda^{-1/2}\text{Op}\left (\mathcal{C}_{12}^* \sigma_p(\Lambda^{1/2})\right)$ is also principally unitary and since it has a positive symbol, it has to have a principal symbol $1$,  i.e., 
\be{C12}
\mathcal{C}_{12}^* \sigma_p(\Lambda^{1/2}) = \sigma_p(\Lambda^{1/2}).
\ee
Let $J(t)$, $t\in[0,t_0]$, be a Jacobi field along the geodesic $\gamma$ connecting a pair of conjugate points $(x,y) = (\gamma(0),\gamma(t_0))$ close to $(p_1,p_2)$, vanishing at $0$ and $t_0$, where $t_0$ is the length of that geodesic. Then \r{C12} would be true when the length of $D_t J$ is the same at the two conjugate points. More precisely, as follows from \cite{MonardSU14}, 
\be{C12'}
\sigma_p(\Lambda^{-1/2})\ \mathcal{C}_{12}^* \sigma_p(\Lambda^{1/2})(x,\xi) = |D_tJ(t_0)|^{1/2}/ |D_tJ(0)|^{1/2},
\ee
for $\xi$ conormal to $\dot \gamma(0)$. The ratio on the right-hand-side of \eqref{C12'} is not equal to $1$ in general, but in some symmetric cases it will be. Indeed, $J$ is given by $J = b \dot\gamma_\perp$ ($\dot \gamma_\perp$ is the vector obtained by rotation of $\dot \gamma$ by $\pi/2$) for some function $b(t)$, and the covariant derivative is then $D_t J = \dot b \dot\gamma_\perp$. The function $b$ satisfies
\[
\ddot{b} + k(\gamma(t)) b = 0, \qquad b(0) = 0, \qquad b(t_0) = 0
\]
where $k$ is the Gaussian curvature and the conjugate point occurs at $t_0$. Thus the change in length of $\dot b$, and therefore $D_t J$, between the two conjugate points can be found from
\be{b0}
\dot b(0)^2 - \dot b(t_0)^2 = \int_0^{t_0} \dot k(\gamma(s)) b^2(s) ds.
\ee
If there is a symmetry so that say $k(\gamma(s)) = k(\gamma(t_0 - s))$ (true for example along geodesics initially tangent to the direction of the waveguide in the cases we consider), then 
\[
|\dot b(0)| = |\dot b(t_0)|
\]
but it is clear that generically this is not true.  
\end{remark}

\begin{remark}\label{remark_1}
By Egorov's theorem, we can ``commute'' one of the \PDO s in (b) of Theorem \ref{thm1} with the FIO $F_{21}$ or $F_{12}$, respectively. Then we get that $F_{21}=K_{21} U_{21}$, where $U_{21}$ is principally unitary in $H^{-1/2}$ as in (a), and $K_{21}$ is a zeroth order \PDO\  with a principal symbol at $(p_1,\xi^1)$ given by $\kappa(p_1,v_1)/ \kappa(p_2,v_2)$, see \r{vxi}. We have a similar property  for   $F_{12}$.
\end{remark} 

As a corollary, we characterize the principal symbols of $F^*_{12}F_{12}$ and  $F^*_{21}F_{21}$, where the star is the $L^2$ adjoint according to our convention. This corollary should be compared to \r{F*eq}, and we note that it makes more precise the failure of principal unitarity on $L^2$ discussed in Remark \ref{remark_1.1}.
\begin{corollary}\label{cor_2.1}
Let $[0,t_0]\ni t\mapsto \gamma$ be the unit speed geodesic issued from $(x,\xi_\perp/|\xi_\perp|)$, where $(x,\xi)\in V^1$, and let $\gamma(t_0)$ correspond to the conjugate point near $p_2$. If $J(t)$ is a Jacobi field along $\gamma$ vanishing at $t=0$ and $t=t_0$, then for $\kappa=1$, 
\begin{align}\label{C21a}
\sigma_p\left(F^*_{12}F_{12}\right)(x,\xi) &= \sigma_p\left(F_{12}F^*_{12}\right)(x,\xi)= |D_tJ(t_0)|/ |D_tJ(0)|,\\
\sigma_p\left(F^*_{21}F_{21}\right)(x,\xi) &= \sigma_p\left(F_{21}F^*_{21}\right)(x,\xi)= |D_tJ(0)|/ |D_tJ(t_0)|.  \label{C21b}
\end{align} 
\end{corollary}
\begin{proof}
Both sides  of \r{LFL}  are $L^2$ principally unitary, call them $U_{21}$. Then 
\[
F_{21} = \text{Op}\left (\mathcal{C}_{12}^* \sigma_p(\Lambda^{-1/2})\right ) \Lambda^{1/2} U_{21}
\]
modulo lower order operators, as above. Therefore, applying Egorov's theorem, we get
\[
F_{21}^* F_{21} = U_{12}\text{Op}\left (\mathcal{C}_{12}^* \sigma_p(\Lambda^{-1})\right ) \Lambda U_{21} = \Lambda^{-1} \text{Op}\left (\mathcal{C}_{21}^* \sigma_p(\Lambda)\right ),
\]
again, modulo lower order operators. By \r{C12'}, this proves our claim for $F_{21}^* F_{21}$.  Similarly, we get
\[
F_{21} F_{21}^* = \text{Op}\left (\mathcal{C}_{12}^* \sigma_p(\Lambda^{-1})\right ) \Lambda,
\]
which proves the corollary for $F_{21} F_{21}^*$, thus finishing the proof of \r{C21b}. The proof of \r{C21a} follows by reversing the direction of $\gamma$; in fact we used that argument above already. 
\end{proof}

The symbols \r{C21a} and \r{C21b} can be computed when $\kappa$ is not constant as well, as in Theorem~\ref{thm1}(b) and the variable weight would contribute elliptic factors of order zero. 

The practical implications are the following. 
Let $\kappa=1$ first. Let us say that $f_1$ is singular in $V^1$ and $f_2$ is not but we do not know this since we know $Xf$ only. Then by \r{cancel}, the microlocal kernel of $X$ in $V^1\cup V^2$ is given by $f=f_1+f_2$ of the form 
\be{2.6}
F_{21}f_1 +f_2 \in C^\infty  \quad 
 \Longleftrightarrow \quad  f_1 +F_{12}f_2 \in C^\infty,
\ee
i.e., of all $f=(\Id-F_{21})f_1$ with $f_1$ singular in $V^1$. 
We can think of $f_2= -F_{21}f_1$ as a mirror image of $f_1$ (using the SAR terminology, see, e.g., \cite{SU-SAR}) which contributes the same singularity to $Xf$ as does $f_1$. In other words, $f=f_1$ and $f=-F_{21}f_1$ would produce the same $Xf$ up to a smooth function because their  difference is in the microlocal kernel.  Those two distributions have the same microlocal  $H^{-1/2}$ strength and in that sense the ``artifact'', if we take $-F_{21}f_1$ as a reconstructed image, would have the same ``norm'' (in fact, we have a family of seminorms). Any microlocal reconstruction of $f=f_1$  would be a linear combination 
\be{ker}
\textrm{reconstructed}\  f= f_1+(\Id-F_{21})h_1
\ee
with some $h_1$ singular in $V^1$. Then the artifact $h_1-F_{21}h_1$ consists of two parts microlocally supported in $V^1$ and $V^2$, respectively, and they have the same $H^{-1/2}$ strengths. 

For general $\kappa$, \r{2.6} still holds but the unitarity statements need to be modified according to Remark~\ref{remark_1}. We have that $f_1$ and $f_2:=-F_{21}f_1$ are still indistinguishable by $X$ in terms of their singularities. On the other hand, their strengths in $H^{-1/2}$ are proportional to the weights there. 

\begin{theorem}[\cite{MonardSU14}]\label{thm_xstar}
With the notation and the assumptions above, 
\be{thm2:eq}
\begin{split}
X^*Xf &= N_1\left( f_1 +  F_{12}f_2\right) \qquad \text{microlocally in $V^1$},\\
X^*Xf &= N_2\left( f_2 +  F_{21}f_1\right) \qquad \text{microlocally in $V^2$},
\end{split}
\ee
where $N_1=X_1^*X_1$ and $N_2=X_2^*X_2$ are \PDO s with principal symbols 
\be{symb1}
\frac{2\pi}{|\xi|}\left( |\kappa(x,\xi_\perp)|^2  +   |\kappa(x,-\xi_\perp)|^2        \right),
\ee
near $V^k$,  $k=1,2$, respectively. 
\end{theorem}

In \r{symb1}, we gave the principal symbol of $N_k$ without the restriction of the directions of the geodesics to be in a small angle near $\dot\gamma_0$. In our situation, one of the terms is always zero, depending on the orientation of $\xi$. 

\subsection{Microlocal analysis of $X$  near a single geodesic in both directions} \label{sec:bothdirections} Let $n=2$. Assume now that $\mathcal{M}$ consists of two connected components $\mathcal{M}_\pm$ corresponding to small neighborhoods of the directed geodesics $\gamma_0 (\pm t)$. If the weight $\kappa$ is not even in its second variable, this gives us extra information which can be used for recovery of singularities. As shown in \cite{MonardSU14}, the following heuristic argument can be made precise: microlocally, to resolve $\WF(f)$  at $(p_1,\xi^1)$ and $(p_2,\xi^2)$   we are solving a system with a matrix
\be{Q}
Q := 
\begin{pmatrix}
\kappa(p_1,v_1) & \kappa(p_2,v_2)\\
\kappa(p_1,-v_1) & \kappa(p_2,-v_2)
\end{pmatrix}
\ee
and a right-hand side $(g_+,g_-) := (Xf|_{ {\mathcal M}_+},Xf|_{ {\mathcal M}_-})$. If $Q$ is invertible, i.e., if $\det Q\not=0$, this can be done. The resulting solution puts the following microlocal weights on $(g_+,g_-)$:
\[
\begin{split}
f_1 &= \frac{ \kappa(p_2,-v_2)g_+ -\kappa(p_2,v_2)g_-  }{\kappa(p_1,v_1) \kappa(p_2,-v_2) - \kappa(p_2,v_2)\kappa(p_1,-v_1) };\\
f_2 &= \frac{ -\kappa(p_1,-v_1)g_+ +\kappa(p_1,v_1)g_-  }{\kappa(p_1,v_1) \kappa(p_2,-v_2) - \kappa(p_2,v_2)\kappa(p_1,-v_1) }.
\end{split}
\]
If $\kappa$ is an attenuation weight as in \r{P1}, then 
\[
\det Q =\Bigg(  \exp \Big \{ -2\int_{ \gamma_{[p_1,p_2]} }a \Big \}- 1 \Bigg)  \exp\Big \{ -\int_{\gamma_0\setminus\gamma_{[p_1,p_2]}}a\Big \},
\]
where $\gamma_{[p_1,p_2]}$ represents the segment of $\gamma_0$ with endpoints $p_1, p_2$ and $\gamma_0\backslash \gamma_{[p_1,p_2]}$ is its complement. If the attenuation is positive, then $\det Q<0$ and the singularities of interest are recoverable. 

Next, we recall what happens if we try the adjoint as an attempt for an inversion.  

\subsection{Analysis of $X^*X$ with even weight} 
Assume that $\kappa$ is an even function of the direction, which happens for example when 
$\kappa=1$. Then $Xf$ near $\gamma_0(-t)$ does not provide any new information. 
If there are no conjugate points, $F_{12}$ (and $F_{21}$, and $f_2$) do not exist, and $N^{-1}X^*$  is a parametrix for $X$, as explained in Section~\ref{sec:singleGeodesic}.  If there are conjugate points as in the theorem, notice first that $N_1$ and $N_2$ are just localized versions of $4\pi \Lambda^{-1}$ up to lower order terms, see \r{symb1}.  
 Therefore, we first apply $(4\pi)^{-1}\Lambda$ to $X^*Xf$ which gives us the microlocal reconstruction
 \[
 \tilde f_1 = f_1+F_{12}f_2, \qquad \tilde f_2 = f_2+F_{21}f_1. 
 \]
Then $F_{12}f_2$ appears as an artifact added to $f_1$ in the microlocal region $V^1$; and similarly in $V^2$. If $f_2=0$, as in our numerical examples, then the reconstruction in $V^1$ is correct. On the other hand, in $V^2$, we get the artifact $F_{21}f_1$ while by assumption, there are no singularities there. More precisely, in this particular case ($f_2=0$),
\be{f2=0}
\begin{split}
N_1^{-1}X^*Xf &= f_1  \ \, \qquad \text{microlocally in $V^1$},\\
N_2^{-1} X^*Xf &=   F_{21}f_1 \quad \text{microlocally in $V^2$},
\end{split}
\ee
with the artifact $F_{21}f_1$ having the same strength as the true image $f=f_1$ when $\kappa=1$ (see also Remark~\ref{remark_1} for general even $\kappa$). Note that the error $F_{21}f_1$ is not in the microlocal kernel, so the ``reconstruction'' is not one of the possible ones. In fact, when $\kappa=1$, $X$ applied to the right hand side of \r{f2=0} gives us $2Xf$ microlocally up lower order terms instead of $Xf$. 
 Therefore, the ``$X^*X$ inversion'', which can be viewed as a backprojection,  fails and it does not even provide  the  solution up to an element of the microlocal kernel. 

In Section~\ref{sec3}, we  analyze the kind of artifacts we get when we  apply the popular Landweber  method when the attenuation is zero. In principle, any such reconstruction, due to inevitable small errors coming from the discretization, etc., would reconstruct $f$ up to some element in the microlocal kernel described in \r{cancel}. That element however, depends on the reconstruction method. 

\subsection{Analysis of $X^*X$ with non-constant weight} Let $\kappa$ be variable and not necessarily even now. If it is even in its second variable, or more generally if $\det Q=0$ near $((p_1,v_1), (p_2,v_2))$, we have the same microlocal behavior as above. If $\det Q\not=0$ however, then $Xf$ known near $\gamma_0(-t)$ provides  extra non-redundant information, as explained above. In that case, $X^*X$ can be viewed as a $2\times2$ matrix valued operator: acting on $(f_1,f_2)$ and localized near $\gamma_0(t)$ and near $\gamma_0(-t)$, respectively. An elementary computation based on the theorem above shows that the ``$X^*X$ inversion'' still provides artifacts even when $\det Q\not=0$; in which case a stable recovery is actually possible. Therefore, we need other reconstruction methods when $\det Q\not=0$; for example when the attenuation is positive. This phenomenon is illustrated in Example~\ref{ex2} and Example~\ref{ex5}, where the artifacts are present in the ``$X^*X$ inversion'' but not in the final one. 

\subsection{Three and more conjugate points} \label{sec_3pts}
Assume now that there are three or more conjugate points $p_1,\dots, p_m$ along $\gamma_0$. Let $f=f_1+\dots+f_m$ with $f_j$ having wave front set near $p_j$ and codirections conormal to $v_j:= \dot\gamma_0$ at $p_j$. Then we can do the same kind of analysis as above but we have two equations (because we have two directions along $\gamma_0$) for $m\ge3$ unknowns. 

More precisely, let $X_{\pm,j}$ be $X$ localized near $p_j$, $j=1,\dots,N$ for geodesics in a neighborhood of $\gamma_0 (\pm t)$. Then microlocally, we get the system
\[
\begin{split}
    X_{+,1}f_1+ X_{+,2}f_2+ \dots + X_{+,m} f_m &= Xf|_{{\mathcal M}_+},\\
    X_{-,1}f_1+ X_{-,2}f_2+ \dots + X_{-,m} f_m &= Xf|_{{\mathcal M}_-}.
\end{split}
\]
The rank of that system is at most $2$ which implies (microlocal) non-uniqueness. For example, given $f_3,\dots, f_m$, then  we can solve for $f_1$ and $f_2$ if the attenuation is not trivial so that the matrix $Q$ is non-singular; and if the matrix is singular, then one can only recover $f_1$ or $f_2$ provided all the other $f_j$'s are given.  This means that we can never resolve all singularities in this case, even in the presence of attenuation. A numerical simulation of this kind is presented in Example~\ref{ex3}. 

\subsection{Dimensions $n\ge3$} \label{sec_2.7} 
In dimensions $n\ge3$, less is known about this problem. First, we may still have a microlocal kernel. We may take a metric on a 2D domain with conjugate points and add, say a third dimension $x^3$ and $(dx^3)^2$ to the metric. Then we get a product manifold and reduce the analysis to the 2D case, see also \cite{SU-caustics}. 

In case of two conjugate points, for the local problem ($Xf$ known near a single geodesic), 
the structure of $X^*X$ is still given by \r{thm2:eq} with FIOs  $F_{12}$ and $F_{21}$ FIOs of order $-(n-2)/2$ with a Lagrangian given by the conormal bundle of the conjugate locus (as a set of pairs), see \cite{SU-caustics, Holman2015}. They  may not be associated to canonical graphs anymore. If they are, they are of negative order as operators mapping Sobolev spaces to Sobolev spaces, and then the singularities can be recovered by the principal \PDO\ part in \r{thm2:eq} even to infinite order by iterations. A necessary and sufficient condition for those FIOs to be associated to local graphs is that the Hessian of the exponential map be non-degenerate where the differential vanishes, see  \cite{SU-caustics, Holman2015} for more details. We do not know however if there are metrics satisfying that condition but in \cite{SU-caustics}, we showed that magnetic geodesics for the Euclidean metric satisfy it.  

As mentioned in the Introduction, there are other ways to recover the visible singularities if $Xf$ is known on a set larger than a neighborhood of a single geodesic. Each non-zero covector $(x,\xi)$ could possibly be resolved by $Xf$ known near a geodesic through $x$ normal to $\xi$. There is a $n-2$ dimensional variety of such geodesics which provides more freedom compared to the 2D case, where there is only one undirected and two directed. If one of those geodesics has no conjugate points, we can resolve $\WF(f)$ from $Xf$ at $(x,\xi)$. If this is true for all $(x,\xi)\in T^*M^\text{int}\setminus 0$, we call $M$ complete and all singularities are stably recoverable even if two or more conjugate points might exist on some geodesics. A 3D example of this sort is presented in Figure~\ref{3Dfig}. 

In \cite{UV:local}, under the assumption of existence of a strictly convex foliation, it is shown that one can recover $f$ from $Xf$ in a stable way. In particular, one can recover all singularities stably. Each singularity is not necessarily recovered by a neighborhood of all geodesics normal to it. The recovery is based on layer stripping, which, for example would recover some of two fixed  conjugate singularities first, and then the other one can be recovered because the first one is already known in equations \r{cancel}. If the latter is conjugate to a third one, then we can recover that one as well, etc.  

\section{Artifacts in the Landweber reconstruction} \label{sec3}

Let us see what happens if we use the Landweber iteration method for numerical recovery. First, assume  $\kappa=1$. As explained above, we would expect to reconstruct the function up to some member of the microlocal kernel, specific for that method. 

\subsection{Brief introduction to the method} 
We recall briefly the method, see \cite{S-Yang-Landweber-17}. Let $\mathcal{L}: \mathcal{H}_1\to\mathcal{H}_2$ be a bounded operator between two Hilbert spaces. We want to solve the equation $\mathcal{L}f=m$, with $m$ being the data; in the range of $\mathcal{L}$ or not (if there is noise). We apply the adjoint $\mathcal{L}^*$ and write the equation in the form
\[
(\Id - (\Id-\gamma\mathcal{L}^*\mathcal{L}))f= \gamma\mathcal{L^*}m.
\]
Here, $\gamma>0$ is a certain constant, often chosen experimentally, so that $K:= \Id-\gamma \mathcal{L}^*\mathcal{L}$ is a contraction, hopefully a strict one. Then we solve the equation above by a Neumann series 
\be{N}
f = \sum_{k=0}^\infty (\Id-\gamma\mathcal{L}^*\mathcal{L})^k \gamma\mathcal{L^*}m, 
\ee
truncated in practice by some criterion. The scheme is
\be{scheme1}
f^{(0)}=0,\quad f^{(k)} = f^{(k-1)} -\gamma \L^*(\L f^{(k-1)}-m), \quad k=1,2,\dots. 
\ee
If $\mu\ge0$ is the largest stability constant for which
\be{stab}
\mu\|f\|_{\mathcal H_1} \le \|\mathcal{L}f\|_{\mathcal H_2}
\ee
(i.e., $\mu^2$ is the bottom of $\spec(\mathcal{L}^*\L)$), then $K$ is a strict contraction if and only if 
\be{stabcond}
0<\gamma<2/\|\mathcal{L}\|^2 \quad\text{and}\quad 0<\mu. 
\ee
It is convenient to extend the notion of stability by restricting $f$ in \r{stab}  to $(\Ker \L)^\perp$ if $\L$ is not injective. Then we call the problem {\em stable} if \r{stab} holds for $f\perp\Ker\L$ with some $\mu>0$. Clearly, all terms in \r{N} stay in  $(\Ker \L)^\perp$. For practical purposes, a very small $\mu>0$ (relative to $\|\L\|$) creates instability, as well; a well known fact in numerical analysis since then the condition number $\|\L\|/\mu$ would be large. 

As mentioned in the previous paragraph when the conditions \eqref{stabcond} are satisfied, $K$ is a strict contraction and so the Neumann series converges uniformly and exponentially to the minimum norm solution of $\L^* \L f = \L^* m$. When the condition $\mu > 0$ fails (i.e. when $\mu =0$), there may still be convergence. Indeed, in the case $\mu = 0$, the Neumann series will still converge to the minimum norm solution of $\L^* \L f = \L^* m$ provided $m \in \mathrm{range}(\L) \bigoplus \mathrm{ker}(\L^*)$. However, for a generic set of $m \notin  \mathrm{range}(\L) \bigoplus \mathrm{ker}(\L^*)$ the iterates are unbounded, and even when there is convergence the speed depends on $m$ and may be very slow. Despite all of this, one can still take a truncated series as an approximate reconstruction with some stopping criterion. 

\subsection{Landweber inversion of $X$} 
\subsubsection{Setup} 
Assume $n=2$, and that $g$ is a non-trapping metric in $M$ with conjugate points. We use full data, i.e., $Xf$ known for all geodesics. By the analysis above, to recover singularities in $V^1$ and $V^2$, only geodesics near $\gamma(t)$ and $\gamma(-t)$ could possibly help, therefore allowing full data does not change the microlocal recovery or the lack of it analyzed in section \ref{sec:theory}. Note that in our examples, one can actually prove that $X$ is injective using the analytic microlocal results in \cite{SU-JAMS, SU-AJM}, for example. Stability depends on the weight however.  

To use the Landweber iteration, we must choose proper spaces first. We view $X$ as an operator $X:  L^{2}(M)\to H^{1/2}(\mathcal{M})$. To avoid dealing with non-local operators, we write the equation $Xf=\psi$ as
\[
(-\Delta_g)^{1/2}\chi X^*Xf= (-\Delta_g)^{1/2}\chi  X^*\psi,
\]
where $\chi$ is a smooth cutoff vanishing very close to $\bo$ and equal to $1$ away from a larger neighborhood. We work with $f$'s supported in $\{\chi=1\}$. We think of $-\Delta_g$ as the Dirichlet realization of the Laplacian in $\Omega$ when taking the square root. 
Then $\L=(-\Delta_g)^{1/2}\chi X^*X : L^2(M)\to L^2(M)$. In the iteration, we need $\L^*\L =  X^*X\chi(-\Delta_g)\chi X^*X$, and $\L^*m = X^*X\chi (-\Delta_g)\chi X^*\psi$.  All adjoints are in $L^2$ here. There are no square roots of the Laplacian anymore, and the approximation sequence is
\be{scheme2}
f^{(0)}= 0, \quad f^{(k)} = f^{(k-1)}-  \gamma  X^*X \chi (-\Delta_g)\chi   X^* (X f^{(k-1)}-   \psi) .
\ee
Then this scheme is equivalent to minimizing $\|(-\Delta_g)^{1/2}\chi  X^*(Xf-\psi)\|^2_{L^2}$ which is equivalent to minimizing $\|\chi  X^*(Xf-\psi)\|^2_{H^1}$ for $f\in L^2$.  Therefore, we are solving numerically $X^*Xf=X^*\psi$ in $H^{1}$ for $f\in L^2$. 

\subsubsection{Inversion with  data in the range} \label{sec_3.3}
The first non-trivial term in \r{scheme2} is 
\be{first}
f^{(1)} = \gamma    X^*X \chi (-\Delta_g) \chi X^*  \psi,
\ee
where $\psi=Xf$, possibly perturbed by noise. If we write this as 
\[
f^{(1)} =  \gamma X^*X \chi (-\Delta_g)^{1/2}\left[ (-\Delta_g)^{1/2} \chi X^* \right]\psi ,
\]
the operator in the bracket is the ``$X^*X$'' attempt for a parametrix, up to a lower order, which works if there are no conjugate points and $\kappa=1$. We can view it as a back-projection.

In \r{thm2:eq}, microlocally, in $\{\chi=1\}$ and up to lower order, we have
\[
\mathcal{L} = (-\Delta_g)^{1/2} \chi X^*X = 
\begin{pmatrix}
\Id  & F_{21}^{-1}\\
F_{21} & \Id
\end{pmatrix}
\]
where we think of functions of the kind $f=f_1+f_2$ microlocally supported in $V_1\cup V_2$ as vector functions $(f_1,f_2)^T$. Then, with $F:=F_{21}$, 
\[
\mathcal{L}^*  \mathcal{L} = X^* X \chi (-\Delta_g) \chi X^*X = 
\begin{pmatrix}
\Id  &F^* \\
 F^*{}^{-1}& \Id
\end{pmatrix}
\begin{pmatrix}
\Id  & F^{-1}\\
F& \Id
\end{pmatrix}
\]
with the adjoints being  in $L^2$. Therefore, 
\be{L*L}
\begin{split}
\mathcal{L}^*  \mathcal{L} &=  
\begin{pmatrix}
  \Id+ F^* F&F^*+ F^{-1}\\
F^*{}^{-1}+ F&  \Id+ F^*{}^{-1} F{}^{-1} 
\end{pmatrix}.
\end{split}
\ee
If $f_2=0$, as in our numerical examples, we get, microlocally in $V^1\cup V^2$, 
\[
f^{(1)} = \gamma\begin{pmatrix}
  (\Id+ F_{21}^* F_{21})f_1 \\
  (\Id+ F_{12}^*F_{12}) F_{21}f_1 
\end{pmatrix}=  \gamma \begin{pmatrix}
 \Id+  F_{21}^* F_{21}&0 \\
  0&\Id + F_{12}^*F_{12}
\end{pmatrix}\mathcal{L}f.
\]
Therefore, we get an elliptic \PDO\ of order zero (whose symbol can be computed by Corollary~\ref{cor_2.1}) applied to the ``$X^*X$ inversion'' $\mathcal{L}f$. This observation is important since in our numerical simulations, $f^{(1)}$ can give us an idea of $(-\Delta_g)^{1/2} \chi X^*X$. If $\kappa$ is variable, then $f^{(1)}$ can be obtained from those expressions by applying an elliptic operator of order $0$. This allows us to see the ``$X^*X$ reconstruction'' numerically in our tests. As expected, it has artifacts regardless of what $\kappa$ is, i.e., regardless of whether the singularities are recoverable or not, see Figure~\ref{2pts_zero_att} and Figure~\ref{2pts_positive_att}. 

Further iterations $f^{(k)}$ are then linear combinations of four types of terms:  $f_1$ and $f_2$ with elliptic zeroth order \PDO s applied to them; and  $f_1$ and $f_2$ with zeroth order elliptic FIOs with canonical relations $\mathcal{C}_{21}$ and $\mathcal{C}_{12}$ applied, respectively. Since those canonical relations are associated to diffeomorphisms, we can express the FIO terms as zeroth order elliptic \PDO s applied to the ``mirror images'' $F_{21}f_1$ and $F_{12}f_2$, respectively. Therefore, $f^{(k)}$ is a sum of elliptic zeroth order \PDO s applied to $f_1$ and $f_2$ and their ``mirror images''. We can compute explicitly at the level of principal symbols to find that
\[
f^{(k)} = \Big (\Id - (\Id- \gamma(2\ \Id + F_{21}^* F_{21} + F_{12} F_{12}^*))^k\Big )\ (\Id +F_{21}^* F_{21})^{-1} F_{21}^* F_{21}\ (f_1 + F_{12} f_2)
\]
modulo smoother terms microlocally in $V_1$, and
\[
f^{(k)} = \Big (\Id - (\Id - \gamma(2\ \Id + F_{12}^* F_{12} + F_{21} F_{21}^*))^k\Big )\ (\Id +F_{12}^* F_{12})^{-1} F_{12}^* F_{12}\ (f_2 + F_{21} f_1)
\]
modulo smoother terms microlocally in $V_2$. Note that these are indeed pseudodifferential operators of order 0, whose principal symbols may be found from Corollary \ref{cor_2.1} acting on the four terms $f_1$, $f_2$, $F_{21} f_1$, and $F_{12} f_2$.

\subsubsection{Heuristic arguments for the expected reconstruction} \label{sec_323}
We will give a heuristic argument explaining the expected artifacts in the Landweber iteration. Let   $f=f_1+f_2$ be as in Theorem~\ref{thm_xstar}. As we proved above, see \r{2.6},  the microlocal kernel consists of $h_1-F_{21}h_1$ with arbitrary $h_1$'s singular in $V^1$. Assume for a moment that this in an actual kernel. Then the Landweber iteration would recover that solution $f_0\in L^2$ of $Xf_0=\psi\in H^{1/2}$ which is orthogonal to the kernel. 
That orthogonal complement  is given by the kernel of $\Id - F^*_{21}$. On the other hand, by \r{ker}, the microlocal solution set   of $Xf=\psi$ is given by $f_1+f_2 + h_1 - F_{21}h_1$; and the latter belongs to the kernel of $\Id - F^*_{21}$  if and only if $f_1+h_1= F^*_{21}(f_2- F_{21}h_1)$, i.e., when $h_1 = -(\Id+F_{21}^*F_{21})^{-1}(f_1-F_{21}^*f_2)$. 
 Therefore, the solution would be
\be{L1}
\begin{split}
\text{Landweber solution} &= 
f_1 + f_2 -(\Id-F_{21})(\Id+F_{21}^*F_{21})^{-1}(f_1-F_{21}^*f_2) \\
 &= \left[f_1- (\Id+F_{21}^*F_{21})^{-1}(f_1-F_{21}^*f_2)) \right] \\
 &{}\quad+ \left[f_2 + F_{21} (\Id+F_{21}^*F_{21})^{-1}(f_1-F_{21}^*f_2)\right] .
\end{split}
\ee
The terms in the square brackets above are supported microlocally in $V^1$ and $V^2$, respectively. For the error defined as $f$ minus the solution above, we get
\[
\text{Error} = (\Id+F_{ 21}^*F_{21})^{-1}(f_1-F_{21}^*f_2) - F_{21} (\Id+F_{21}^*F_{21})^{-1}(f_1-F_{21}^*f_2).
\]
The error is in the range of $\Id-F_{21}$ as we established in \r{ker}. It 
consists of two parts of equal strength in $H^{-1/2}$, each one being elliptic \PDO s applied to $f_1$, $f_2$ and to their ``mirror images'' $F_{12}f_2$ and $F_{21}f_1$  as defined in section~\ref{sec:singleGeodesic} (since we can write, for example, $F^*_{21}f_2 =( F^*_{21}F_{21}) F_{12}f_2 $). Let $f_2=0$ as in our numerical examples. Then 
\be{L2}
\begin{split}
\text{Landweber solution} &= f_1 -(\Id-F_{21})(\Id+F_{21}^*F_{21})^{-1}f_1 \\
&= \left[f_1- (\Id+F_{21}^*F_{21})^{-1}f_1) \right] + \left[F_{21} (\Id+F_{21}^*F_{21})^{-1}f_1\right],
\end{split}
\ee
and
\[
\text{Error} = Pf_1 - F_{21} Pf_1, \quad P:=- (\Id+F_{21}^*F_{21})^{-1}. 
\]
Therefore, the error in $V^1$ is $Pf_1$, where $P$ is an elliptic \PDO, and in $V^2$, it is $-F_{21}Pf_1$. 
While  arbitrary solutions may contain arbitrary elements of the microlocal kernel, in particular $h$ with $\WF(h)$ not even in $V^1\cup V^2$,  the Landweber method with exact data however gives artifacts which are $Pf_1$ in $V^1$ and its ``mirror image'' with an opposite sign in $V^2$. This can also be explained by the minimal norm requirement --- additional artifacts would increase the norm of the error. If there is noise however, this would change. 

The situation gets simpler if we  minimize $\|X^*(Xf-\psi)\|_{H^{1/2}}^2$ for $f\in H^{-1/2}$. This would require the use of the non-local operator $(-\Delta_g)^{1/2}$, but then we have the advantage that $F_{21}$ and $F_{12}$ are principally unitary in $H^{-1/2}$. Then $P=-1/2$ modulo $\Psi^{-1}$ and we get that the Landweber solution would be 
\be{3.10}
\frac12 \left(f_1+F_{21}f_1\right)+ \frac12 \left(f_2+F_{12}f_2\right).
\ee
Note that this is $1/2$ of the ``$X^*X$ inversion''. Therefore, we get a half of the originals and the other half is transformed into the artifacts $\frac12F_{21}f_1$ and $\frac12F_{12}f_2$. In this case, by \r{L*L}, $\mathcal{L}^*=\mathcal{L}$ microlocally up to lower order, and $(\mathcal{L^*L})^k =2^{(2k-1)}\mathcal{L} $. This implies that the subsequent iterations change the coefficient to approximately $1/2$ of the original $f=f_1$ (plus the artifact in \r{L1}) but do not change the form of the reconstruction much.

Going back to the minimization for $f\in L^2$, note that in our numerical examples, $P$ is close to $P=1/2$ because of the approximate symmetry there, see Remark \ref{remark_1.1}.  The first non-zero term is approximately $\gamma$ times the ``$X^*X$ inversion'', see \r{first}. The numerical behavior we observe is close to that in the previous paragraph, see   Example~\ref{ex4}. 

Those arguments can be made precise and we will only sketch the proof. To solve $Xf=\psi$ microlocally for $f=f_1+f_2$, we seek that solution which is orthogonal to the microlocal kernel of $X$ which can be seen by \r{L*L} to be the same as that of $(\Delta_g)^{1/2}X^*X$; and given by the microlocal kernel of $\Id+F_{12}$. The arguments are the same as before but with errors smooth functions. 

\subsubsection{Functions with high-frequency content} 
The analysis above applies asymptotically to functions which are not necessarily singular  but  have large high-frequency support. Examples are highly concentrated Gaussians or coherent states, see \r{coh}.  The full analysis can be done along the same lines but using the semi-classical calculus \cite{Zworski_book}. We will consider here a special case which we use in our numerical computations. 
If we take any singular $f=f_1+f_2$ as above, we can convolve it with  $\phi_h(x) = h^{-n}\phi(x/h)$ with some $\phi\in C_0^\infty$, $0<h\ll1$. Then using the semiclassical calculus,  one can show that $f_h:= \phi_h* f$ has a semiclassical wave front set $\WF_h(f_h)$ as  $\WF(f)$ but restricted to the dual variable $\xi$ in $\supp\hat\phi$ (which we can take radial). Since $X$ is smoothing on the microlocal kernel \r{2.6}, for any $f^\sharp$ in that set, $Xf_h^\sharp = O(h^\infty)$. 
 Therefore, such an $f_h^\sharp$ is not in $\Ker X$ (which might be trivial) but it is ``almost in the kernel''.  

As shown in \cite{S-Yang-Landweber-17}, the rate of convergence of the Neumann series \r{N} or, equivalently, the sequence \r{scheme2}, depends on the spectral decomposition of $f$ w.r.t.\ the spectral measure related to $|\L|$ defined as the square root of $\L^* \L = X^*X\chi (-\Delta_g)\chi X^*X$. If we denote the spectral representation of $h$ by $\tilde h$, then $\tilde f^{(1)} = \gamma\lambda^2\tilde f$ and a simple calculation, see also \cite{S-Yang-Landweber-17}, yields  
\be{fk}
\tilde f^{(k)} = \left(1-(1-\gamma\lambda^2)^k\right)\tilde f  .
\ee
Note that the multiplier here is very small near $\lambda=0$ ($\sim k\gamma\lambda^2$) and approaches rapidly $1$ when $k$ grows, for every $\lambda>0$, see Figure~\ref{spectral}. The sequence $\tilde f^{(k)}$ converges to $f$ projected to the orthogonal complement of the kernel of $|\mathcal{L}|$ by the Lebesgue dominated convergence theorem (see also \cite{S-Yang-Landweber-17}) but clearly, the rate of convergence of $\tilde f^{(k)}$ restricted for small $\lambda$'s is much slower than the rest. With $f=f_h$, we can  decompose $f_h$ as in the heuristic argument above as a sum of an element of the microlocal kernel  $f_h^\sharp$ and its orthogonal complement $f^\sharp_{h,\perp}$.  
The spectral representative $\tilde f_h^\sharp $ of $f_h^\sharp$ then  will be supported essentially near $\lambda=0$ and its convergence will be very slow instead of being unchanged as in the heuristic argument. The iterations will modify the sequence applied to $f^\sharp_{h,\perp}$ mostly which, modulo $O(h)$ (since our analysis is on the principal symbol level only), would produce an approximate solution of the kind \r{L1}.

\begin{figure}[h!] 
  \centering 
  \includegraphics[trim = 0 0 0 0, clip, height=.15\textheight]{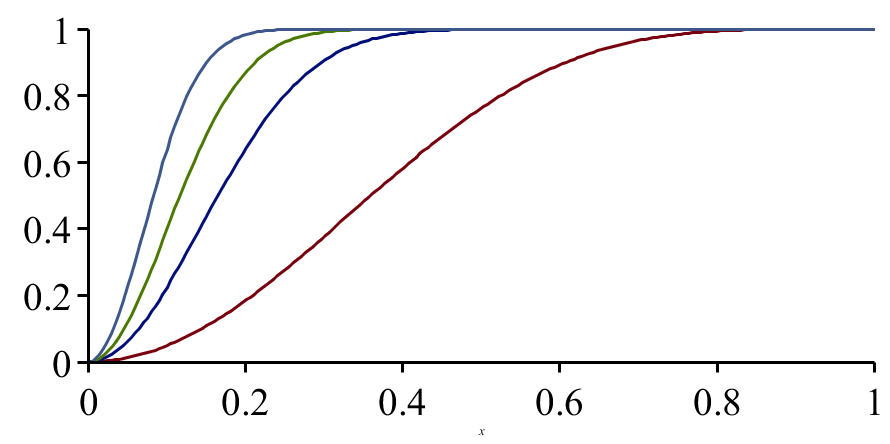} 
  \caption{\small The graph of $\phi_k(\lambda):= 1-(1-\gamma\lambda^2)^k$ as a function of $x= \lambda\sqrt\gamma\in [0,1]$  for $k=5, 25, 50, 100$. The smoother curve corresponds to $k=5$. }
\label{spectral}
\end{figure}

Finally, we want to emphasize that we have two large parameters in our analysis: the frequency $|\xi|$ (or $1/h$) and the number of the iterations $k$. The statements hold by taking $|\xi|$ or $1/h$ large enough first (or taking  the asymptotic) and then taking $k\gg1$ depending on $|\xi|$ or $1/h$.

\subsection{Regularizing property of the Landweber method. Data not in the range} 
The Land\-weber method is known to have certain regularizing properties, see, e.g., \cite{EHN}. We will give some theoretical insight into this more aligned with our analysis. 
Classical regularization methods replace the inversion of an operator which does not have a bounded inverse (or that are not even injective) by an inversion of an operator having a bounded inverse (with a large bound). One typical case is to add to $\L^*\L$ an operator $\eps R$ with $R>0$ and $0\le\eps\ll1$; then we invert $\L^*\L+\eps R$. One could choose $R$ to be a positive power of the Laplacian if we need the regularization effect for large frequencies only. Often, as in Tikhonov regularisation, this is posed as a minimization problem solved by iterations. Other kinds of ``variational" regularization, which may not lead to linear problems, are frequently used as well.  However, in the Landweber case, even though the series diverges for generic perturbed data \cite{EHN,S-Yang-Landweber-17}, the partial sums have a regularization property which we now describe (for a more general treatment of regularisation in iterative methods by early stopping see e.g. \cite{EHN}). 

This property already follows from the analysis in the preceding subsection if the data $g$ is in the range. In the spectral representation, the $f^{(k)}$'s are obtained from $f$ by a multiplication by the filter $\phi_k(\lambda)$, see \r{fk} and Figure~\ref{spectral}. This filter cuts away the small $\lambda$ modes from the spectrum providing regularization. If $\mu>0$, the spectrum does not contain $(0,\mu)$, and for $k\gg1$, $f^{(k)}$ would be very close to $f$. When $\mu=0$, there is no stability but with exact data the iterations would still converge to $f$. We see that they get less and less regularized in that convergence, where by regularizing we mean cutting off some neighborhood of $\lambda=0$. As an example, in cases where $\L$ is a smoothing operator either by a finite degree or infinitely smoothing (the microlocal kernel consists of the whole space then), high frequencies map to such a neighborhood. In the case under consideration, highly oscillatory functions approximately in the microlocal kernel  would map to a small neighborhood of $\lambda=0$ as well. 
 
To analyze data not in the range, as in  \cite{S-Yang-Landweber-17}, write the problem $\L f=m$ with $m$ not necessarily in the range as 
\[
\left(\Id - (\Id - \gamma \L^*\L)\right) f = \gamma |\L| U^*m,
\]
where $\L= U|\L|$ is the polar decomposition of $\L$ \cite{Reed-Simon1}. While $\L f=m$ is inconsistent in general, the equation above is consistent if $\mu>0$. When $\mu = 0$, it is consistent provided $m \in \mathrm{range}(\L) \bigoplus \mathrm{ker}(\L^*)$, and is  equivalent to $|\L|^2 f = |\L|m_*$ with $m_*:= U^*m$. Since $m_*\in(\ker|\L|)^\perp$, the latter equation has unique solution in that space with spectral representation $\tilde{f}$ given by $\tilde f = \tilde{m}_*/\lambda$ if the right hand side is in the Hilbert space; and in general the solution is the unbounded operator of dividing the spectral representation $\tilde m_*$ by $\lambda$. 

Let $f^{(k)}$ be the partial sum in \r{N} with $\infty$ there replaced by $k-1$, as above. Then in the spectral representation, $ f^{(k)}$ takes the form \cite{S-Yang-Landweber-17}
\[
\tilde f^{(k)}= \sum_{j=0}^{k-1} (1-\gamma\lambda^2)^j \gamma\lambda \tilde m_*= g_k(\lambda) \tilde m_*,
\]
with 
\[
g_k(\lambda) := \frac{1-(1-\gamma\lambda^2)^k}\lambda
\]
see Figure~\ref{TAT_with_reflection_numerics_fig1}. The function $g_k$  extends to  a smooth function on $\lambda\ge 0$ and is therefore bounded on the spectrum of $|\L|$ but not uniformly bounded in $k$; its least upper bound grows at least as $\sqrt{k}$.  Its pointwise limit is $1/\lambda$, of course. 
The function $g_k(\lambda)$ can be viewed as a product of the true but unbounded inverse $1/\lambda$ and the filter $\phi_k(\lambda) = 1-(1-\gamma\lambda^2)^k$, see Figure~\ref{spectral}:
\[
g_k(\lambda)  = \phi_k(\lambda)  \frac1\lambda. 
\]
As $k$ grows, this filter cuts off a smaller and a smaller neighborhood of the singular point $\lambda=0$ in a smooth way, acting as a regularizer to the true solution. Unlike the variational regularizers, the filter becomes less and less restrictive with the iterations and in particular if regularization is not needed (when $\mu>0$ is not too small), for large enough $k$ the filter is very close to $1$. On the other hand, getting a good reconstruction depends on choosing well the constant $\gamma$ and the stopping criteria. 
 
\begin{figure}[h!] 
  \centering
  \includegraphics[scale=0.22]{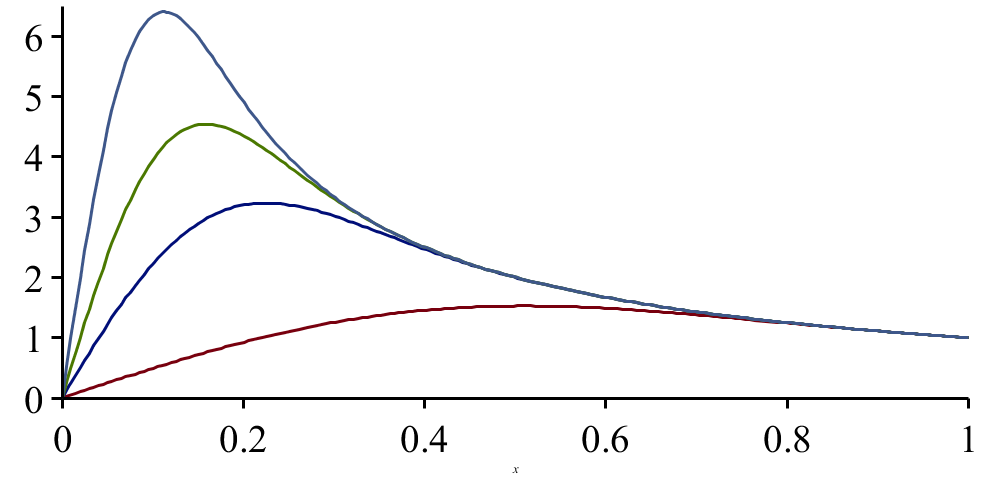}
\caption{\small The functions $g_k(\lambda)$ with $\gamma=1$ and $k=5, 20, 40, 80$. As the number of iterations  $k$ increases, the maximum increases at least as $C_1\sqrt k$ and its location shifts to the left to $\lambda_k\sim C_2/\sqrt k$.}
\label{TAT_with_reflection_numerics_fig1}
\end{figure}

Finally, we want to mention that in the considerations above, we made some idealizations. In numerical inversions, we invert a discretized version of the problem which may not approximate the continuous problem well at discrete frequencies close to the Nyquist one. In particular, a stable problem may have a unstable discretization but still behave in stable way in the inversions \cite{S-Yang-Landweber-17}. The analysis still applies to the discrete problem (ignoring rounding errors at each step) but an effective inversion does not use the adjoint $L^*$ of the discretized $\L$. Instead, at each step, it computes $L^*$ acting on particular element by some kind of backprojection, typically. That operator is close to the matrix $L^*$ but not the same, which creates an additional error. For more details in another inverse problem, we refer to \cite{S-Yang-Landweber-17}. 

\section{Numerical examples} \label{sec:numerics}

We present some numerical reconstructions illustrating the points made in the previous sections. The computational domain is the unit disk, embedded in a cartesian grid $[-1,1]^2$ discretized uniformly into $n\times n$ points with $n=300$. The metrics we use below are conformally Euclidean, of the form $c^{-2}\d x^2$, where the sound speed $c$ takes either of the three forms: 
\begin{align}
    \begin{split}
	c_1(x,y) &= \exp \left( 0.3 \exp \left( -\frac{y^2}{2\sigma_1^2} \right) \right), \qquad \sigma_1 = 0.25, \\
	c_2(x,y) &= \exp \left( 0.3 \exp \left( -\frac{y^2}{2\sigma_2^2} \right) \right), \qquad \sigma_2 = 0.12, \\
	c_3(x,y) &= \exp \left( 0.65 \exp \left( - \frac{x^2 + (y-0.3)^2}{2\sigma_3^2} \right) + 0.65 \exp \left( - \frac{x^2 + (y+0.3)^2}{2\sigma_3^2} \right) \right), \qquad \sigma_3 = 0.25.
    \end{split}
    \label{eq:ci}
\end{align}
$c_1$ and $c_2$ model a ``gutter'' (or waveguide) along the $x$ axis, with no more than pairs of conjugate points along any geodesic for $c_1$, and some triples of conjugate points along near-horizontal geodesics for $c_2$. $c_3$ has two focusing lenses located at $(0,\pm 0.3)$; see Fig. \ref{fig:sample} for sample geodesics. 

\begin{figure}[htpb]
    \centering
    \includegraphics[trim = 93 40 80 30, clip, height=.14\textheight]{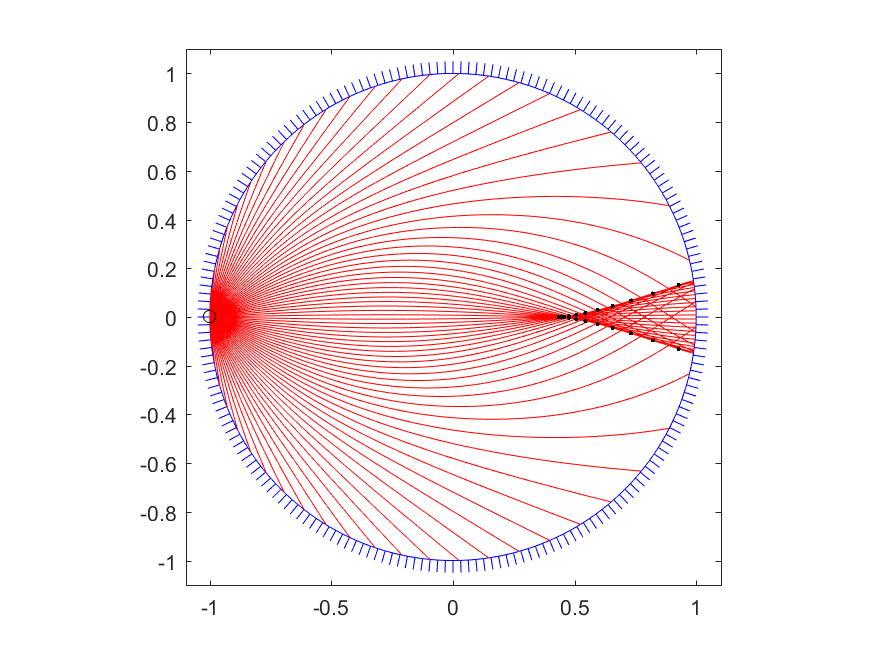}
    \includegraphics[trim = 93 40 80 30, clip, height=.14\textheight]{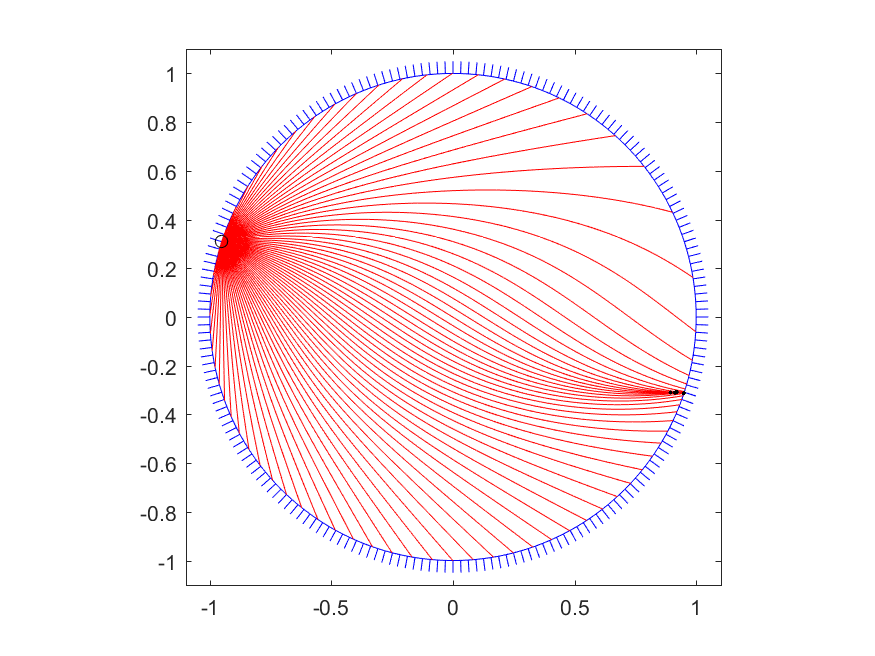}
    \includegraphics[trim = 93 40 80 30, clip, height=.14\textheight]{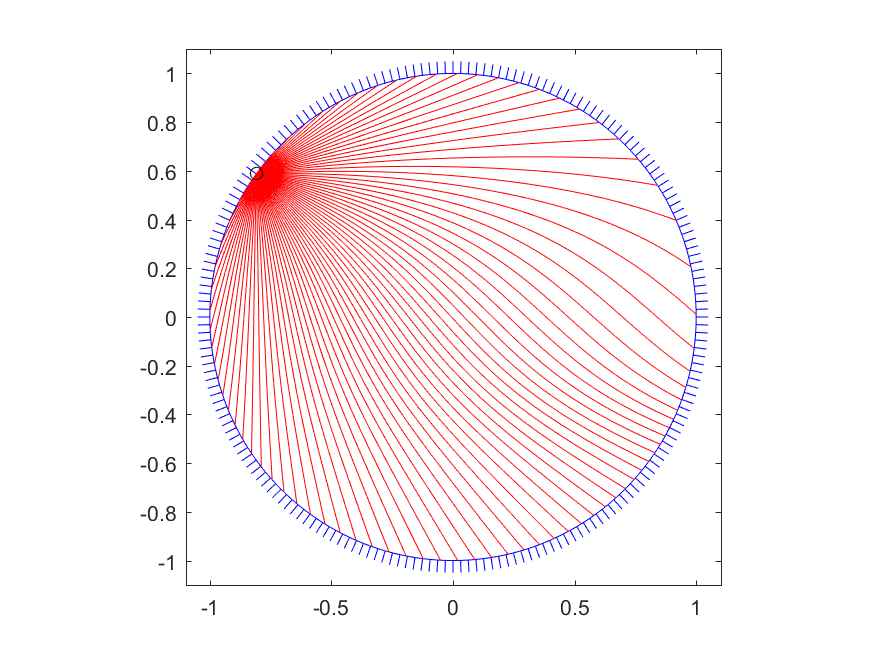}
    \includegraphics[trim = 93 40 80 30, clip, height=.14\textheight]{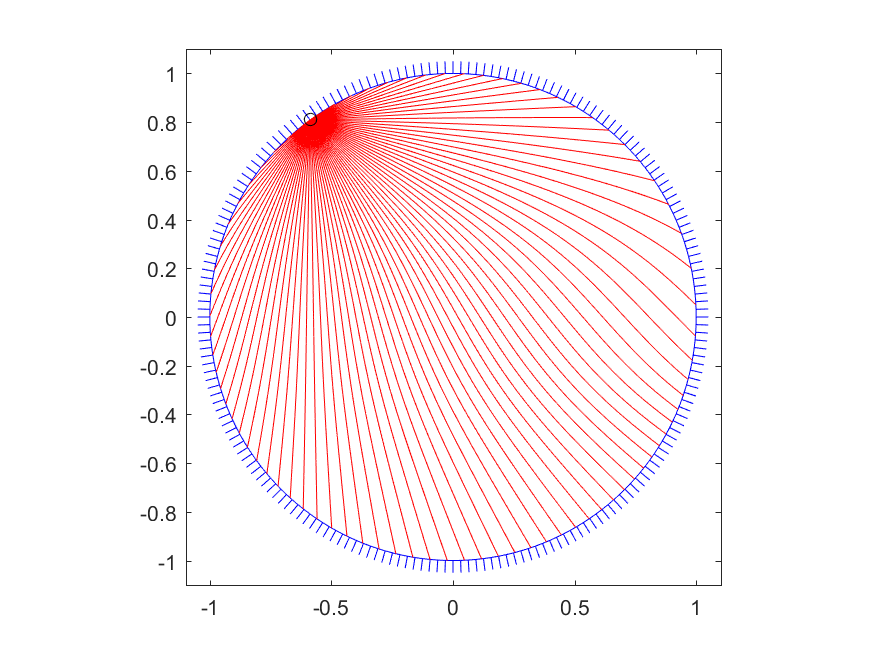}
    \includegraphics[trim = 93 40 80 30, clip, height=.14\textheight]{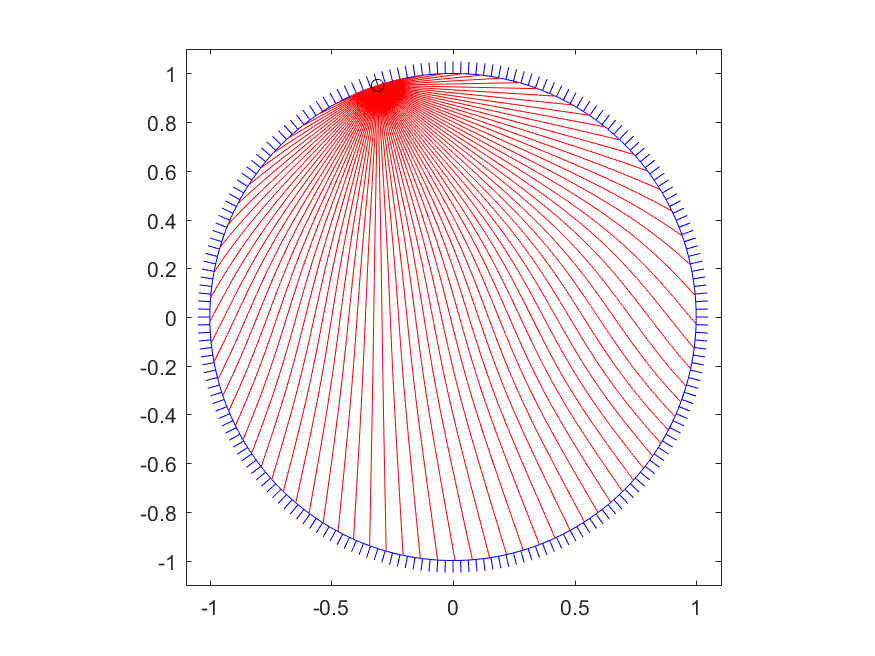}    
    \includegraphics[trim = 93 40 80 30, clip, height=.14\textheight]{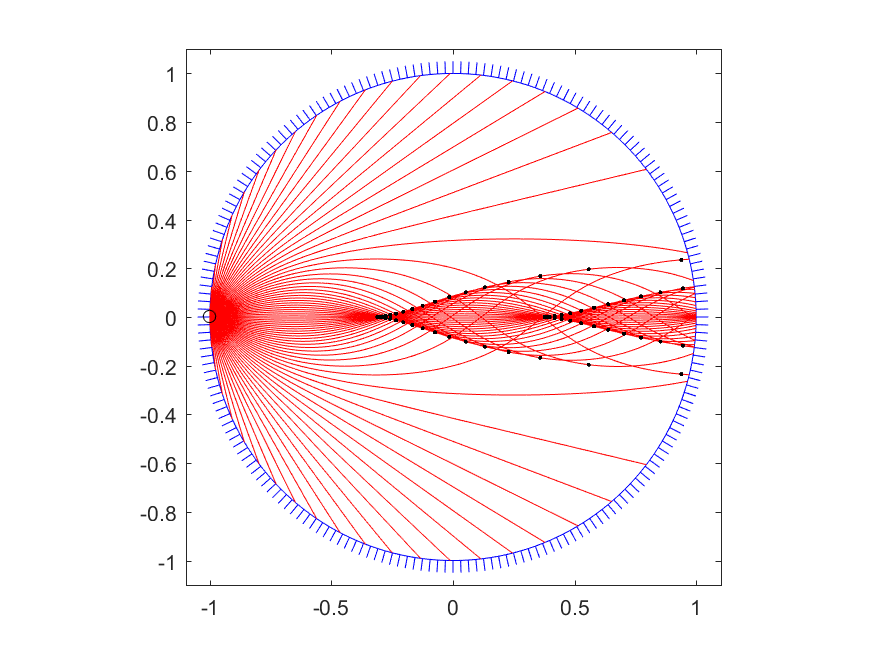}
    \includegraphics[trim = 93 40 80 30, clip, height=.14\textheight]{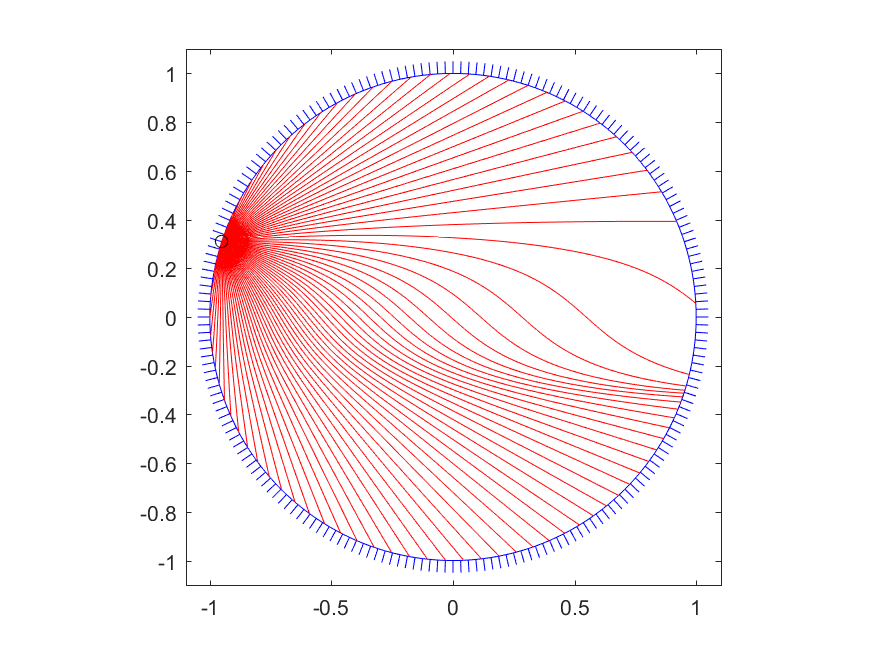}
    \includegraphics[trim = 93 40 80 30, clip, height=.14\textheight]{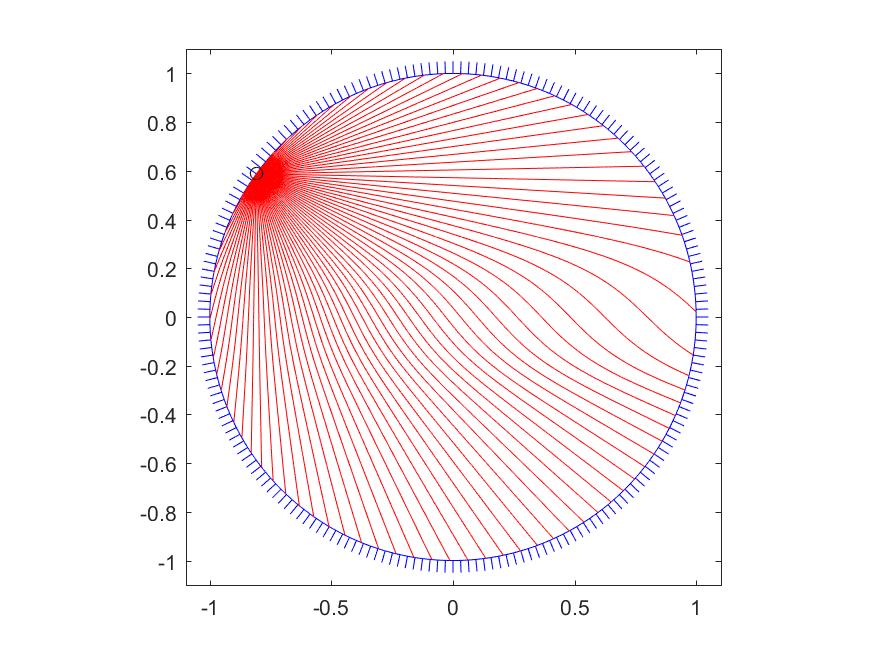}
    \includegraphics[trim = 93 40 80 30, clip, height=.14\textheight]{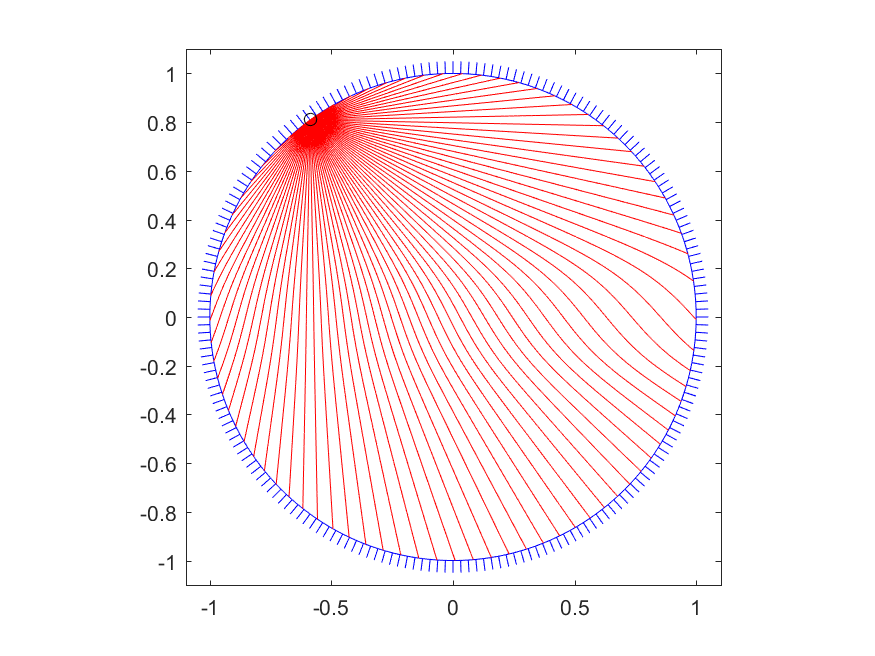}
    \includegraphics[trim = 93 40 80 30, clip, height=.14\textheight]{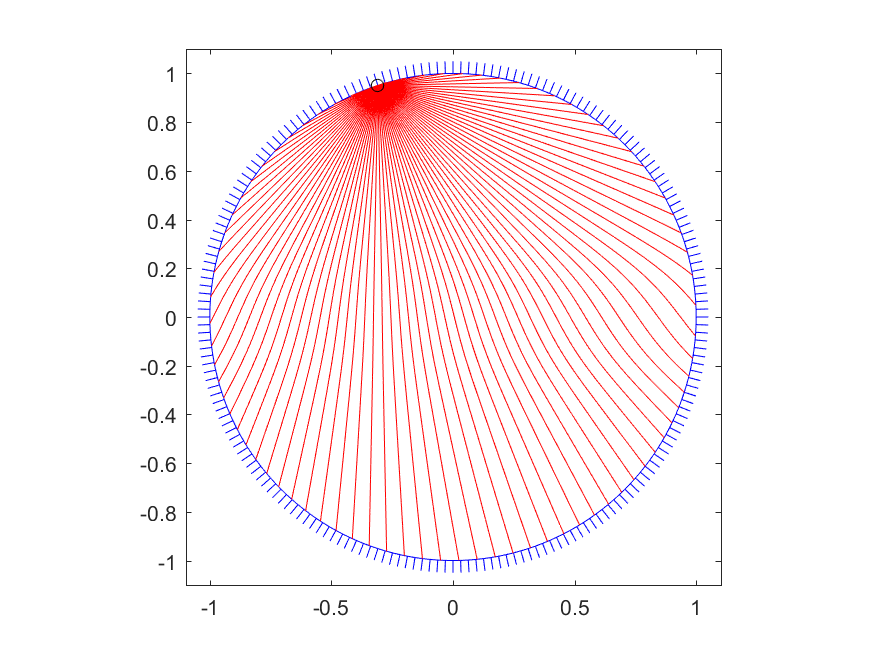}    
    \includegraphics[trim = 93 40 80 30, clip, height=.14\textheight]{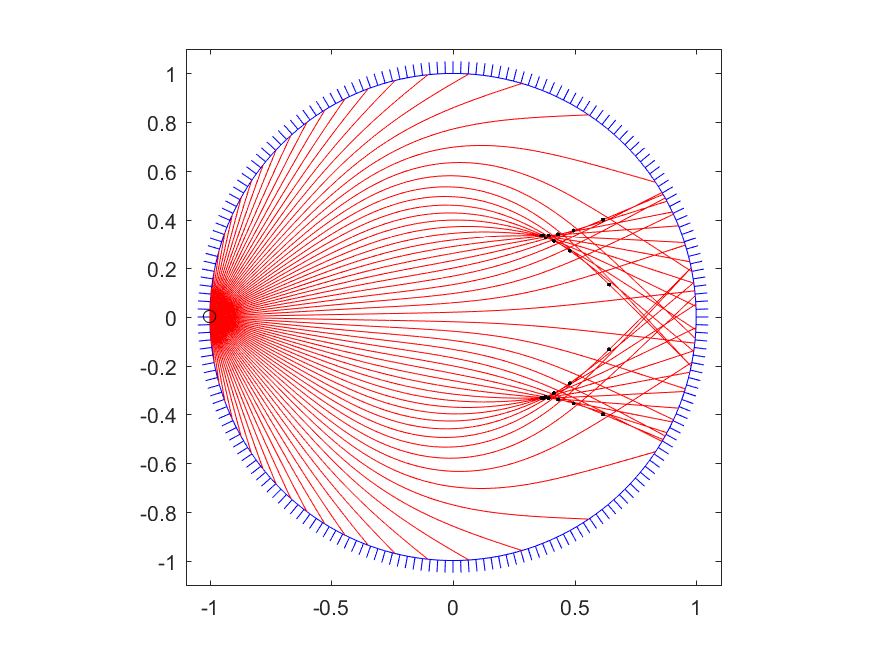}
    \includegraphics[trim = 93 40 80 30, clip, height=.14\textheight]{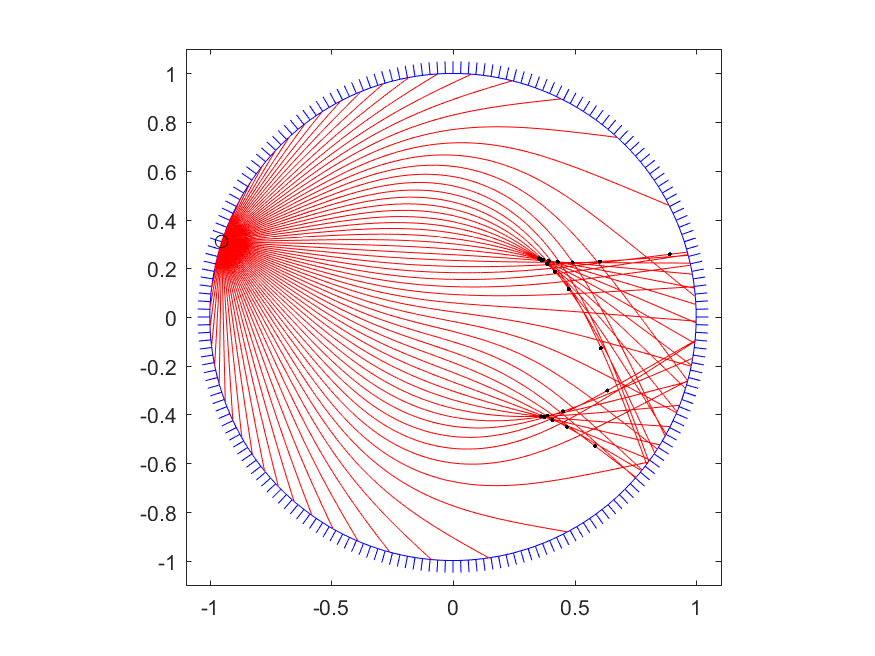}
    \includegraphics[trim = 93 40 80 30, clip, height=.14\textheight]{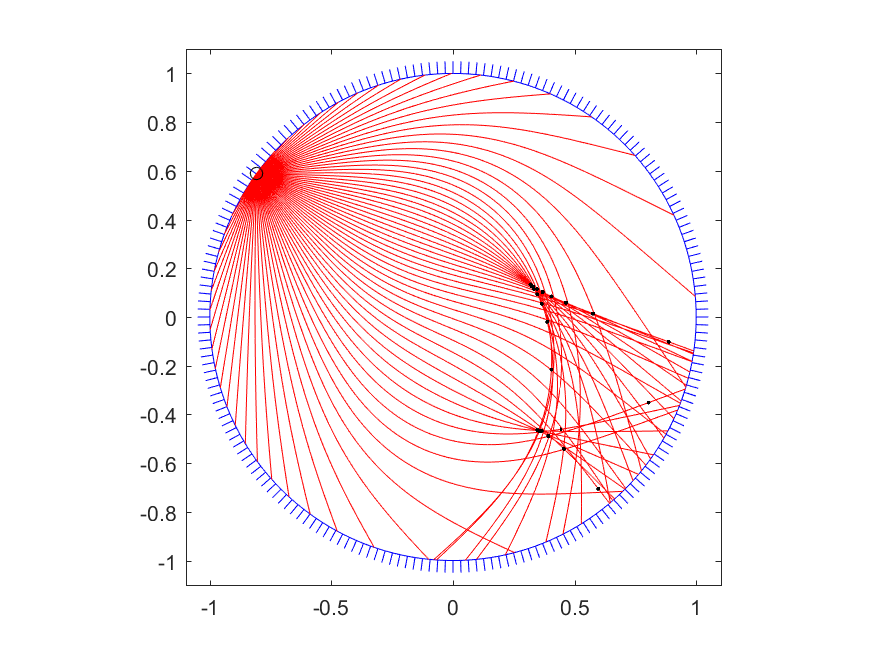}
    \includegraphics[trim = 93 40 80 30, clip, height=.14\textheight]{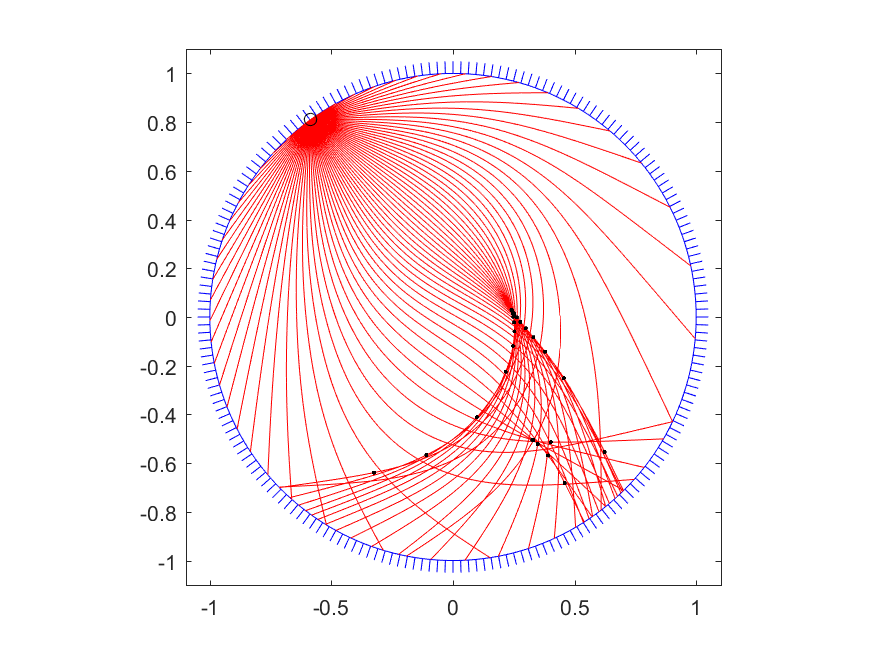}
    \includegraphics[trim = 93 40 80 30, clip, height=.14\textheight]{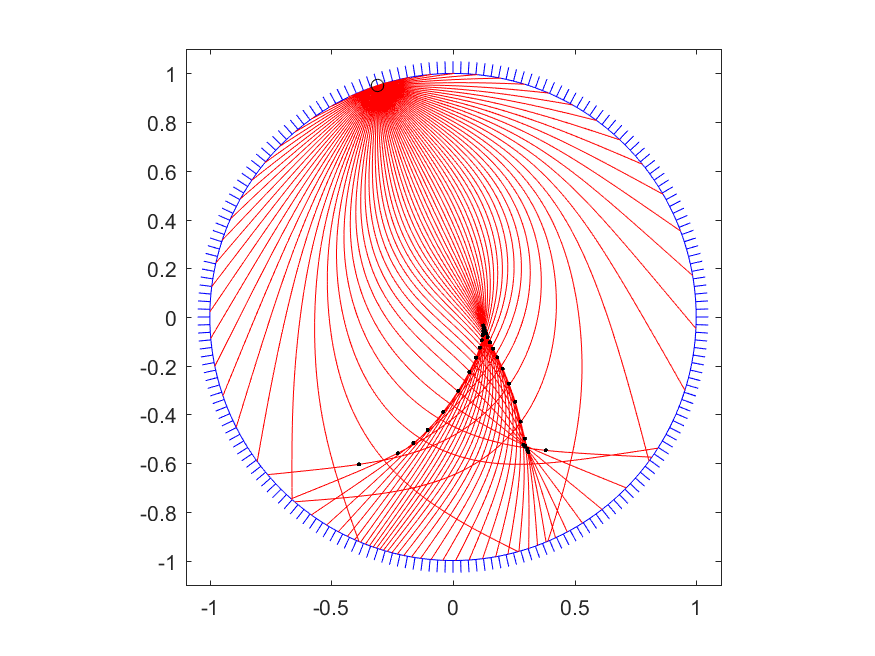}   
    \caption{Sample geodesics for the three metrics used in the examples. For $i=1,2,3$, row $i$ corresponds to speed $c_i$ in \eqref{eq:ci}. The domain is the unit disk. All speeds have reflection symmetry with respect to the $x$ and $y$ axes. }
    \label{fig:sample}
\end{figure}

For numerical reasons, we do not work with phantoms which are actually singular (say, having sharp jumps) because  the discretization in that case is problematic; in particular we can have aliasing even at the sampling stage. Instead, we choose  smooth phantoms with a high enough frequency content approximating singular ones. In all examples below except the 3D one, we run Landweber iterations following the scheme \eqref{scheme2}, where $f^{(k)}$ will always denote the reconstruction after $k$ iterations. For such examples, we discretize forward and adjoint operators as in \cite{Monard14}, and compute Laplacians using finite differences on the cartesian computational grid. 

\begin{example}[Figure~\ref{2pts_zero_att}: speed $c_1$, zero attenuation] \label{ex1}
    The phantom is chosen to be of ellipsoidal shape to make the near horizontal edge longer, where the non-recoverable singularities lie. Those singularities are not recoverable and the purpose of this example is to illustrate that. Note that the rest of the edge is fully recoverable. We show the geodesics issued from the center of the phantom. The geodesics tangent to the near horizontal parts of the edge are close to the one plotted and the artifact is concentrated around the conjugate locus of that point.     
    
    \begin{figure}[h!] 
	\centering 
	\includegraphics[trim = 30 70 10 90, clip, height=.2\textheight]{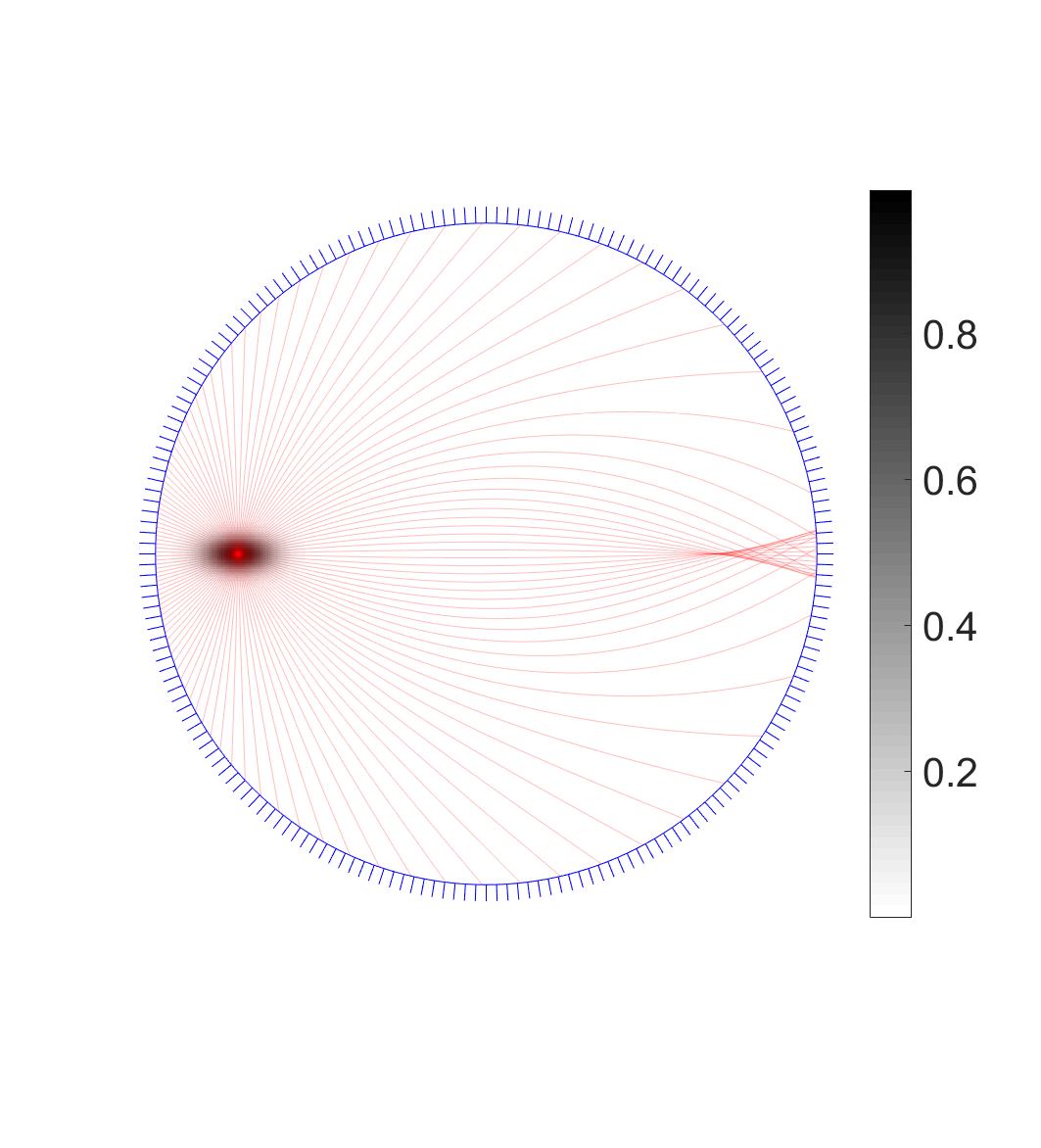} 
	\includegraphics[trim = 30 70 10 90, clip, height=.2\textheight]{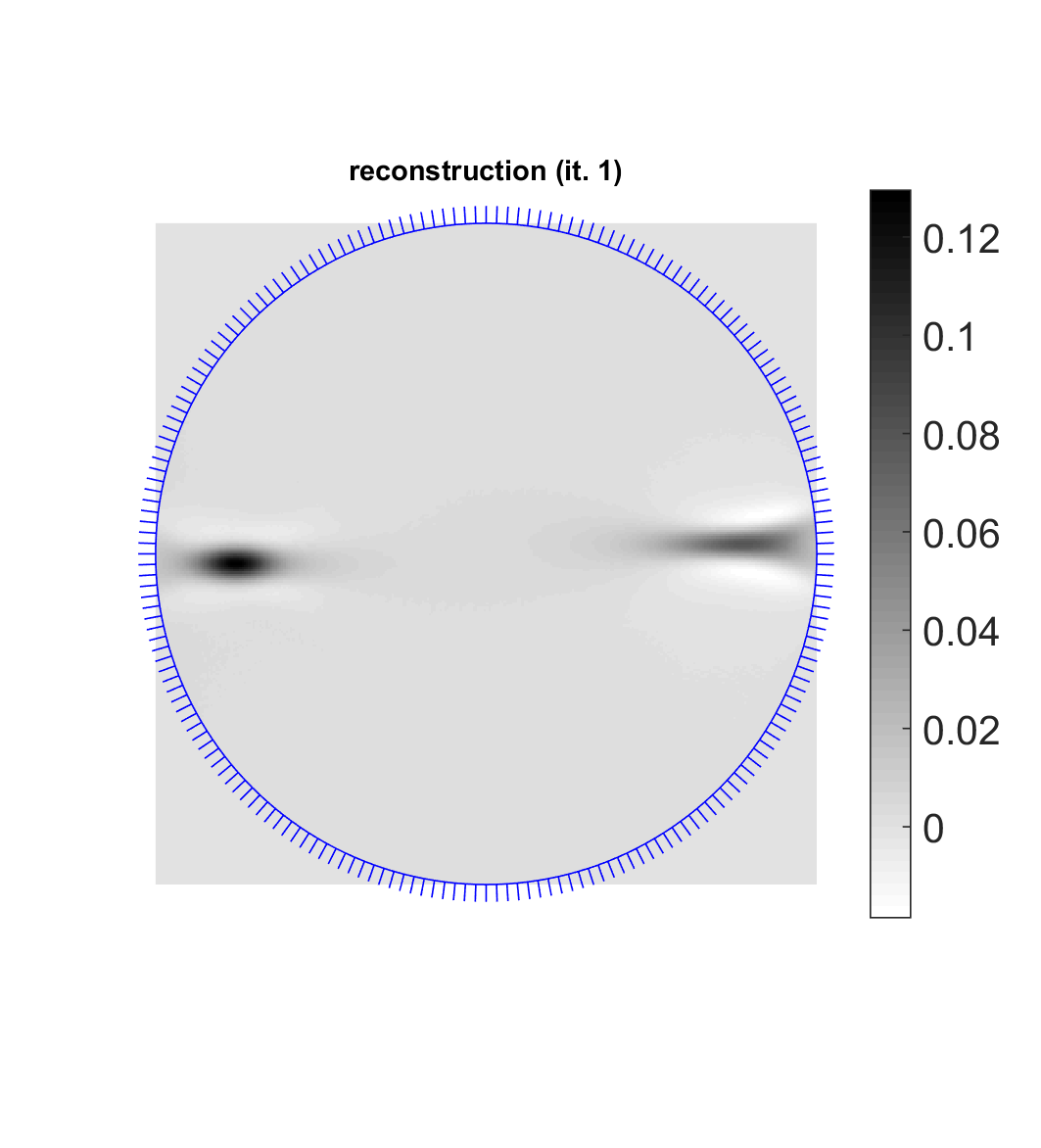}
	\includegraphics[trim = 30 70 10 90, clip, height=.2\textheight]{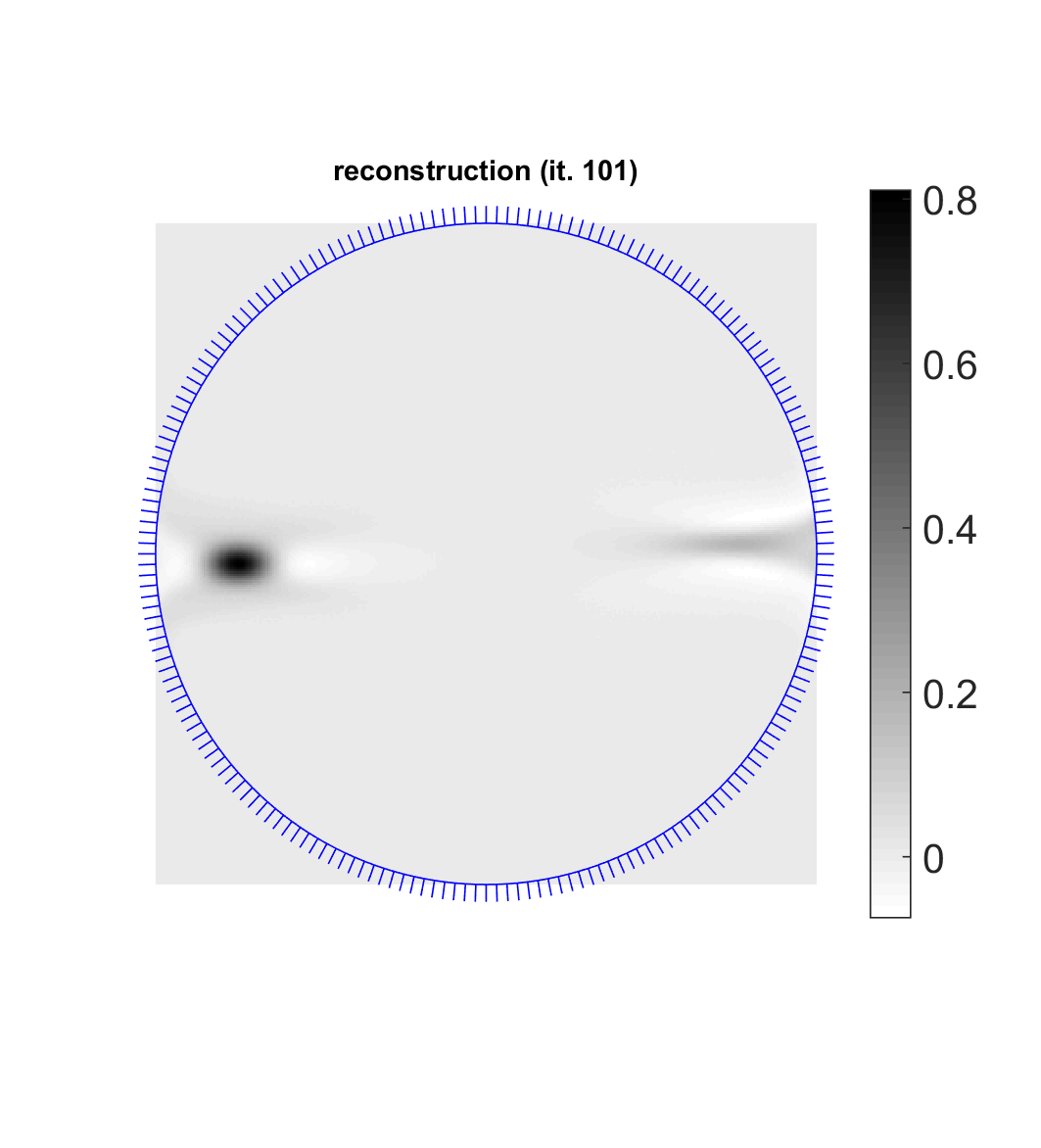}
	\caption{\small Example \ref{ex1}: two conjugate points, zero attenuation. Left to right: true $f$ with geodesics superimposed; $f^{(1)}$; $f^{(101)}$. Artifacts remain at the conjugate locus.}
	\label{2pts_zero_att}
    \end{figure}
    
    In line with the theory, the middle reconstruction on the second row is the first iteration $f^{(1)}$, which is
a zeroth order \PDO\ applied to the ``$X^*X$ inversion'' up to a lower order and must show the artifact $F_{21}f_1$ with a similar \PDO\ applied to it. Because of the symmetry, we are getting approximately a scalar multiple of  the ``$X^*X$ inversion'' and a scalar multiple of the mirror image $F_{21}f_1$. 
    We see that the first iteration (the backprojection) has an artifact which seems to weaken with the iterations but does not disappear. The weakening however is relative to the rest of the edge, which is fully recoverable; not relative to the near horizontal ones! In Example~\ref{ex4} below, this becomes much clearer (note the different scales for $f^{(1)}$ and $f^{(101)}$ there). Also, a plot of the error, not shown here, displays two parts with approximately equal amplitudes. 
\end{example}


\begin{example}[Figure~\ref{2pts_positive_att}: speed $c_1$, non-zero attenuation] \label{ex2}
    Same as Example \ref{ex1} but the attenuation is positive. We see that the first iteration $f^{(1)}$ (the backprojection) has an artifact at the conjugate locus. Iterating, it gets weaker and weaker and almost disappears. This is an illustration of the theoretical possibility of recovering all singularities when the attenuation is positive. 
    
    \begin{figure}[h!] 
	\centering 
	\includegraphics[trim = 30 70 10 90, clip, height=.2\textheight]{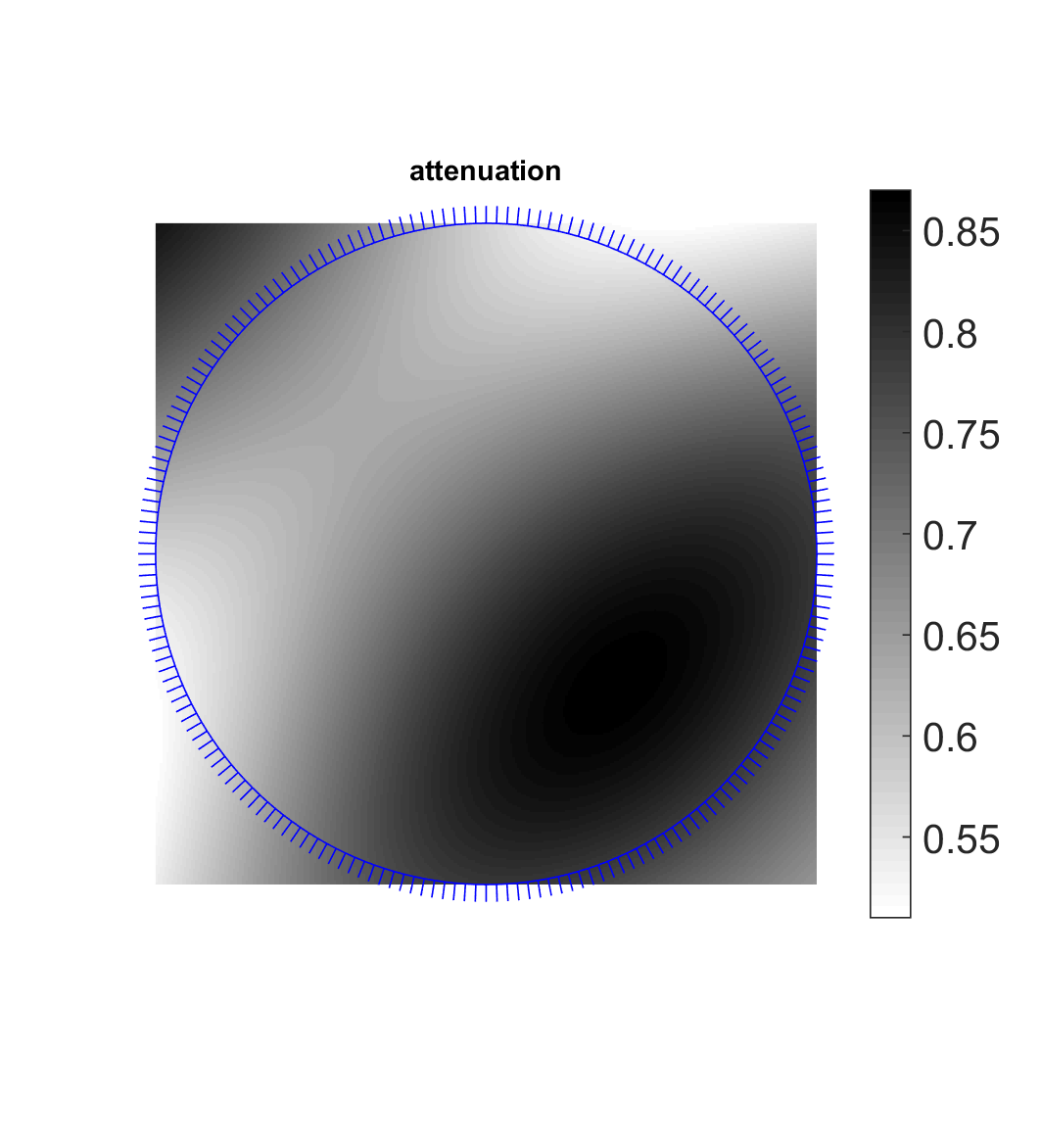}
	\includegraphics[trim = 30 70 10 90, clip, height=.2\textheight]{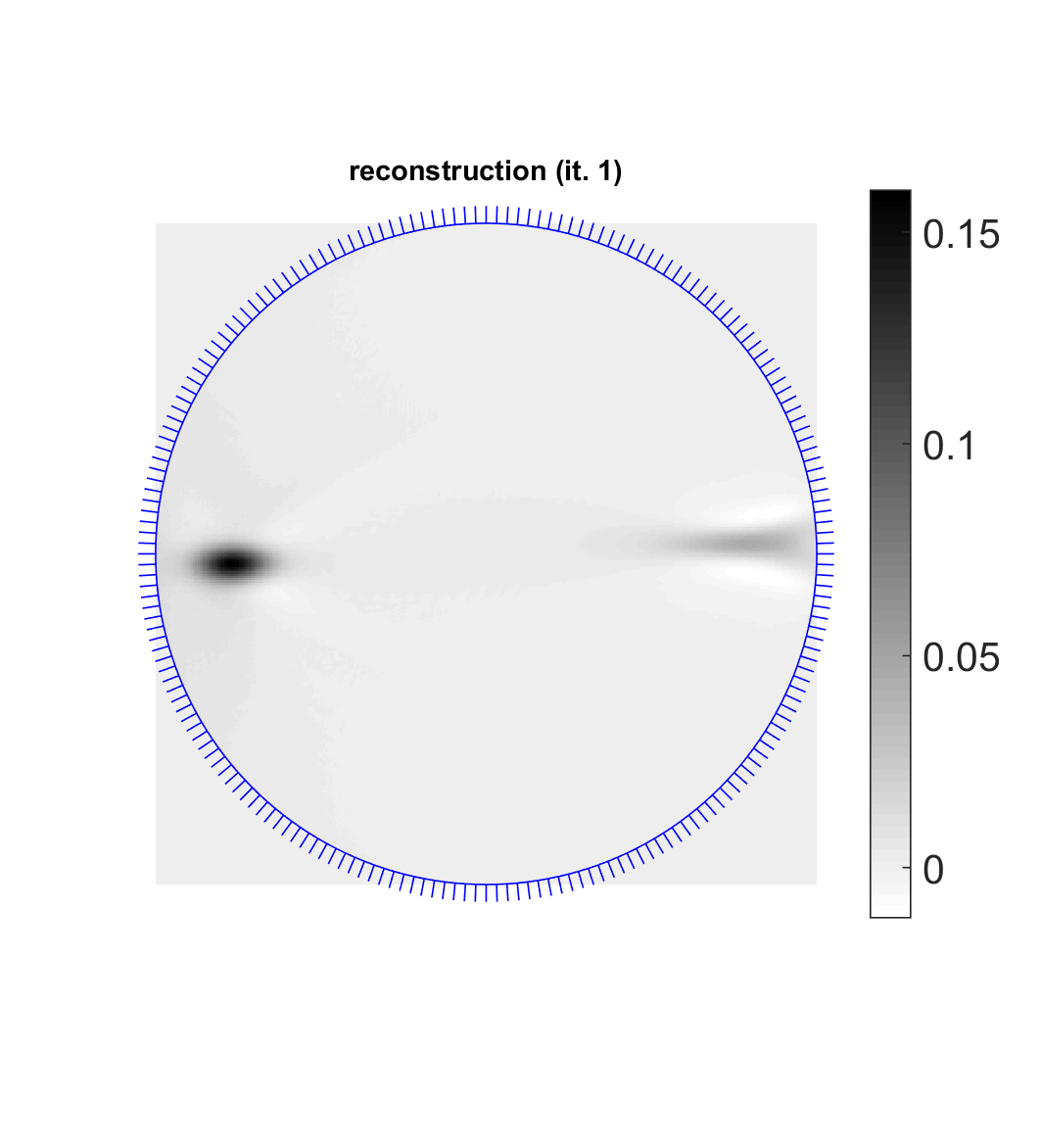}
	\includegraphics[trim = 30 70 10 90, clip, height=.2\textheight]{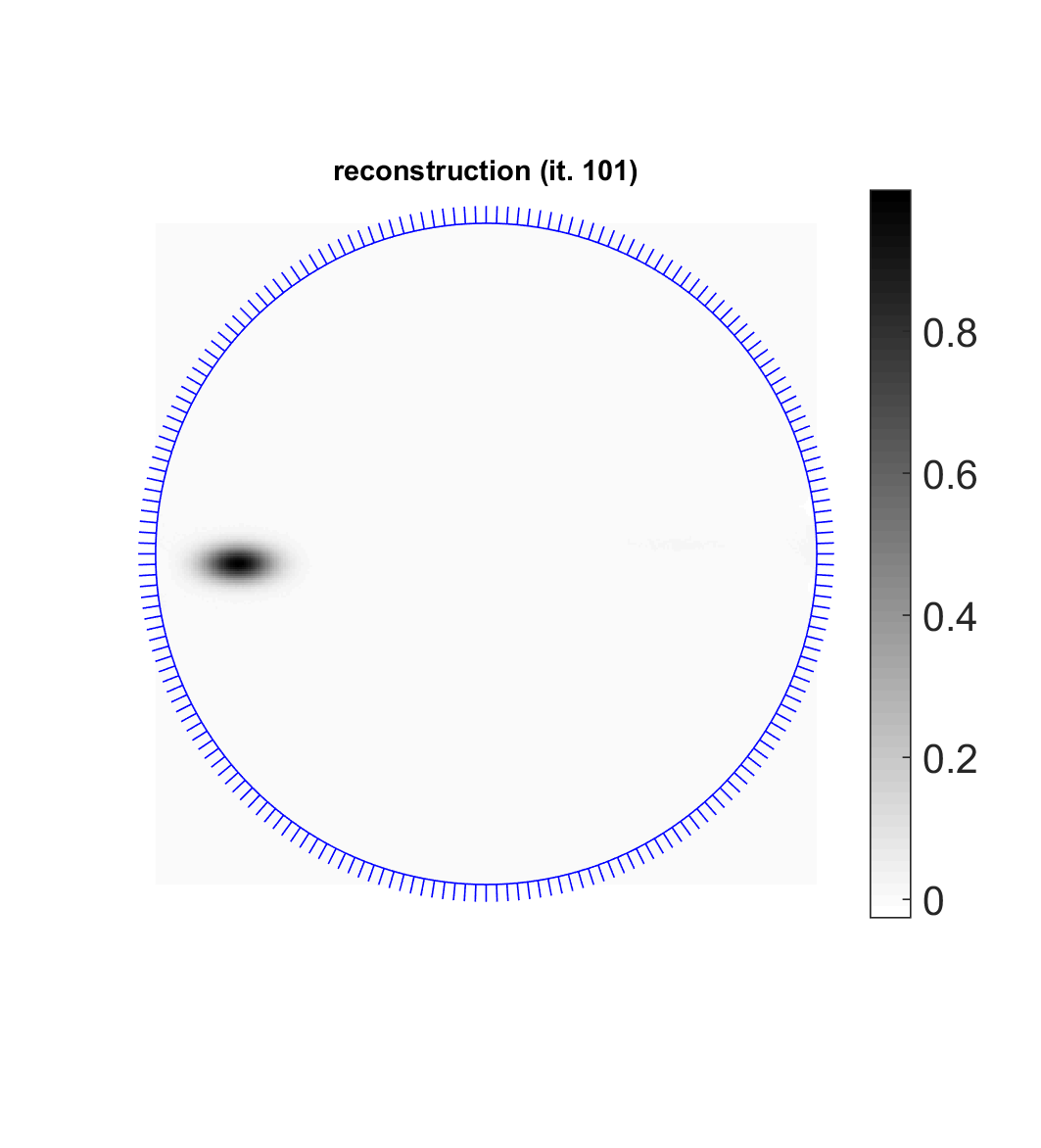}
	\caption{\small Example \ref{ex2}: two conjugate points, positive attenuation. $f$ is as in Fig. \ref{2pts_zero_att}. Left to right: Attenuation $a$; $f^{(1)}$; $f^{(101)}$.
    }
	\label{2pts_positive_att}
    \end{figure}
\end{example}

\begin{example}[Figure~\ref{2pts_coherent_zero}: speed $c_1$, coherent state, zero attenuation] \label{ex4}
    We choose $f$ to be of a coherent state type with singularities well localized in the phase space, i.e., both in space and direction, given by
\be{coh}
f(x,y) = \sin(y/\sigma^2) e^{ -(x^2+y^2)/2\sigma^2}, \quad \sigma=0.1
\ee
on the $[-1,1]^2$ square then  shifted to left by $0.7$ and rotated by $\pi/24$.  The high-frequency content is along edges close to horizontal, and we can consider $f$ as an approximation to a distribution having wave front set along the ray $(x,\lambda\xi) = ((0,0),\lambda (0,1))$ before the shift and the rotation. 
The metric is as above. 
    \begin{figure}[h!] 
	\centering 
	\includegraphics[trim = 30 70 10 90, clip, height=.21\textheight]{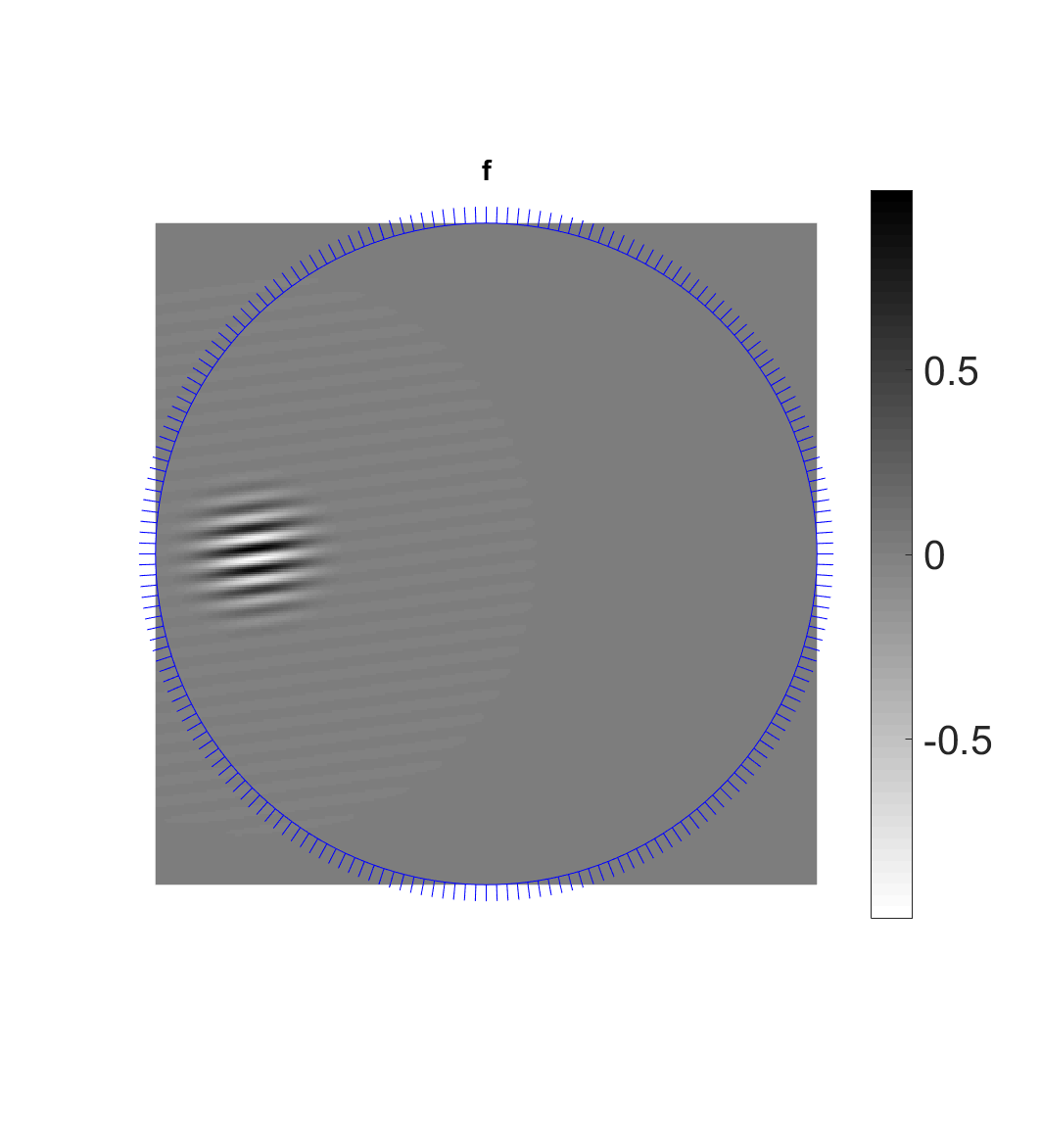} 
	\includegraphics[trim = 30 70 10 90, clip, height=.21\textheight]{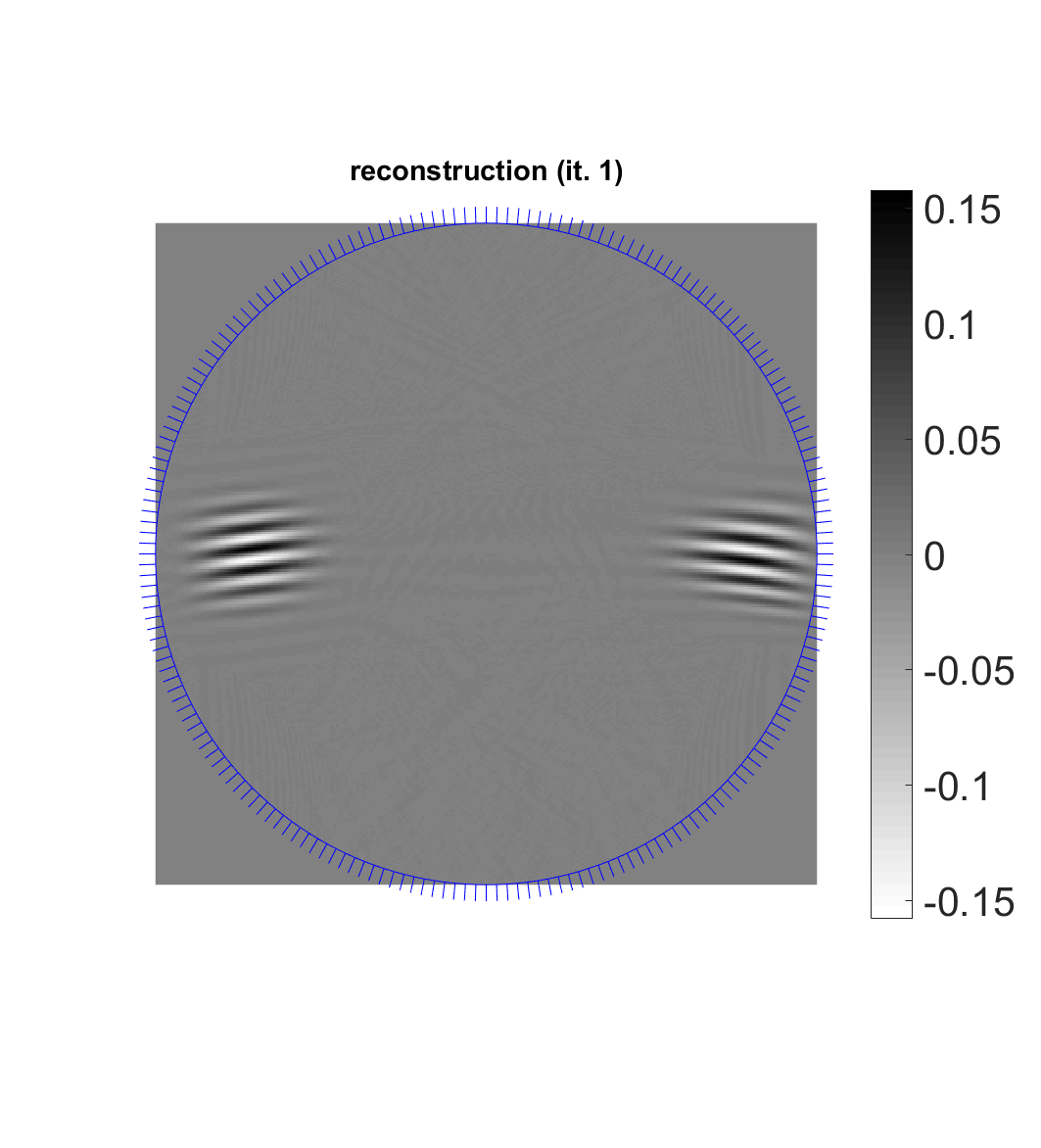}
	\includegraphics[trim = 30 70 10 90, clip, height=.21\textheight]{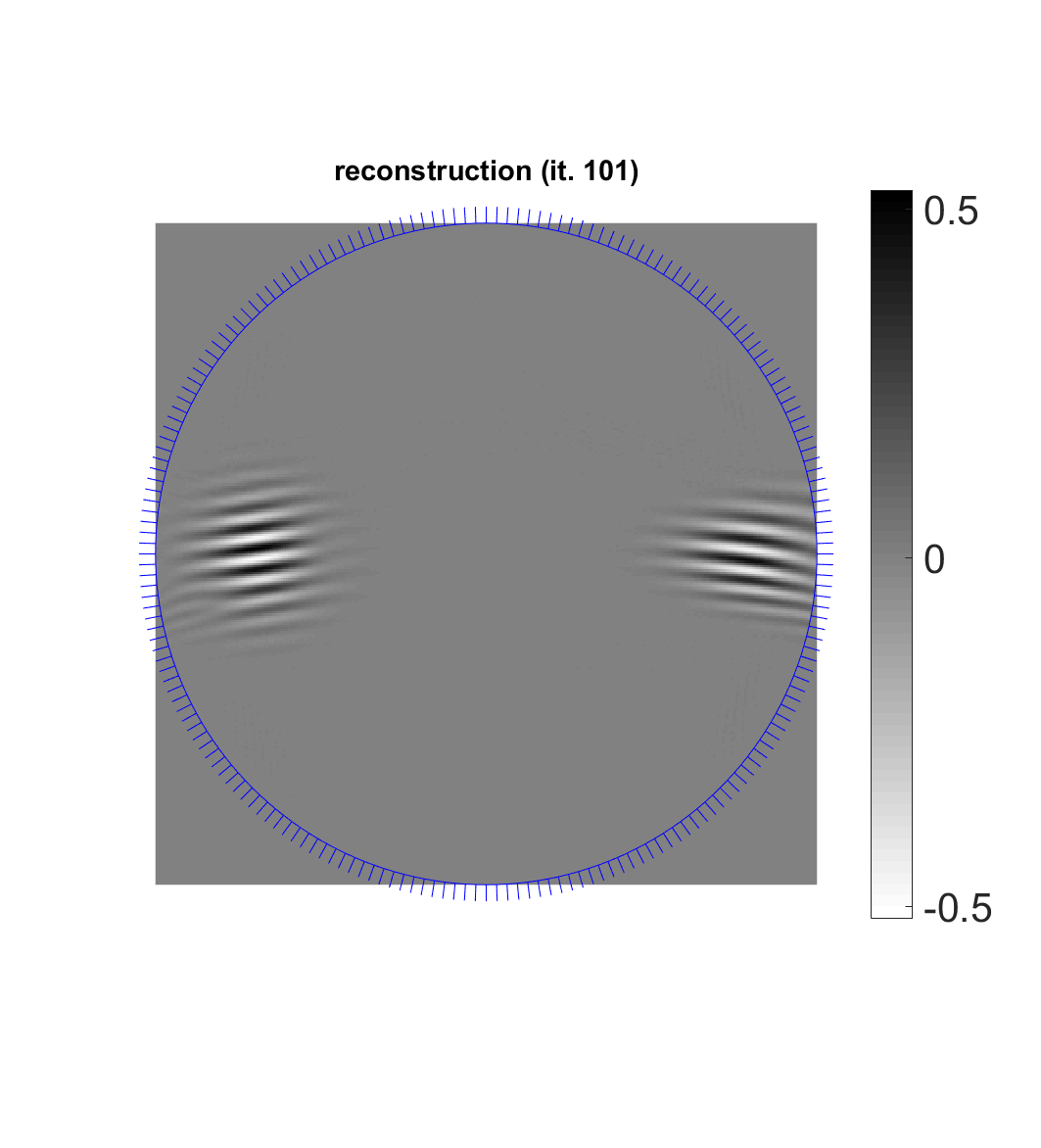}
	\caption{\small Example \ref{ex4}: two conjugate points, zero  attenuation. Left to right: true $f$; $f^{(1)}$; $f^{(101)}$. Artifacts appear in both cases and the iteration mostly scales the image to $1/2$ of the original plus the artifact without changing much anything else. The error, not shown, consists of two parts of approximately equal magnitudes.}
	\label{2pts_coherent_zero}
    \end{figure}
    
    Unlike the previous examples, the singularities are not fully recoverable and the reconstruction should produce artifacts as in Example~\ref{ex1}. The first iteration, ($f^{(1)})$ shows an artifact of equal strength, as expected. Subsequent iterations mostly scale $f^{(1)}$ up until the scaling factor reaches $1/2$. If everything is perfect, including no discretization, the sequence would eventually converge to $f$ but of course, in practical applications, this would not happen due to the instability. If we keep iterating, we start seeing increasing high-frequency noise-like artifacts.

\end{example}

\begin{example}[Figure~\ref{2pts_coherent_positive}: speed $c_1$, coherent state, positive attenuation] \label{ex5}
    We choose $f$ to be a coherent state as in Example~\ref{ex4}. The attenuation is positive. We can recover $f$ well.  This example is similar to Example~\ref{ex2}. The metric is the same as in the previous two examples. 
    \begin{figure}[h!] 
	\centering 
	\includegraphics[trim = 0 50 0 90, clip, height=.21\textheight]{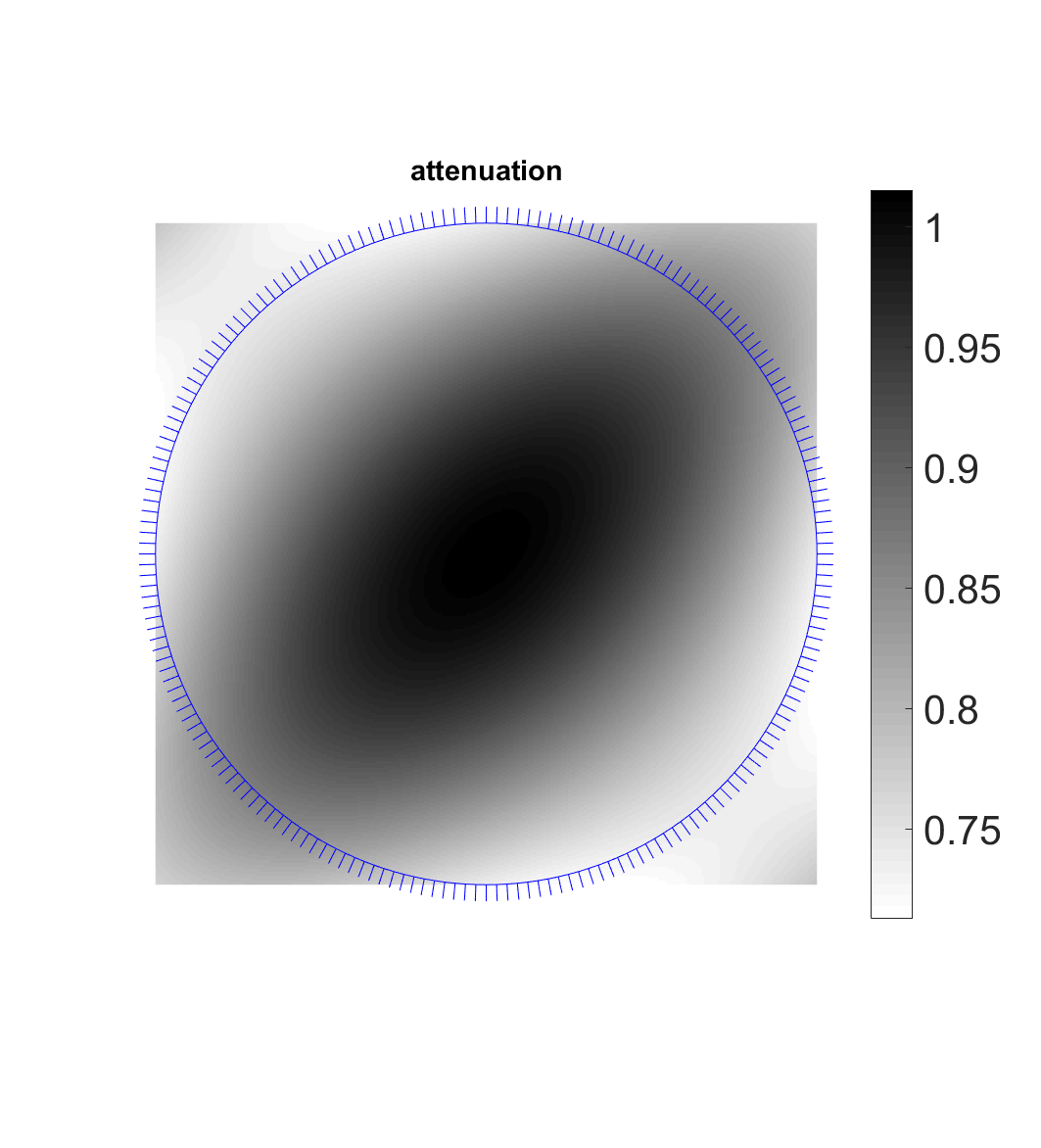}
	\includegraphics[trim = 30 70 10 90, clip, height=.21\textheight]{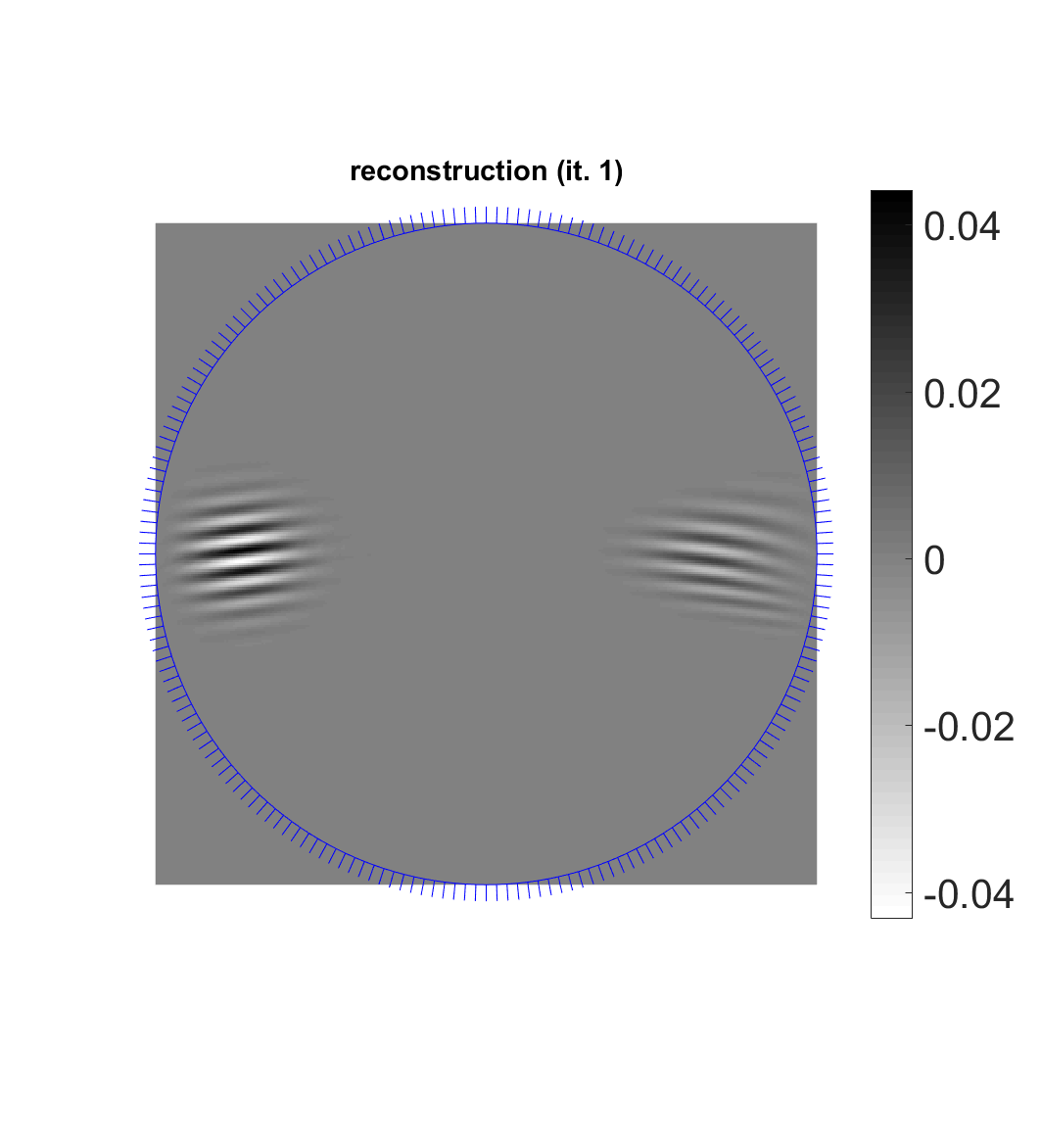}
	\includegraphics[trim = 30 70 10 90, clip, height=.21\textheight]{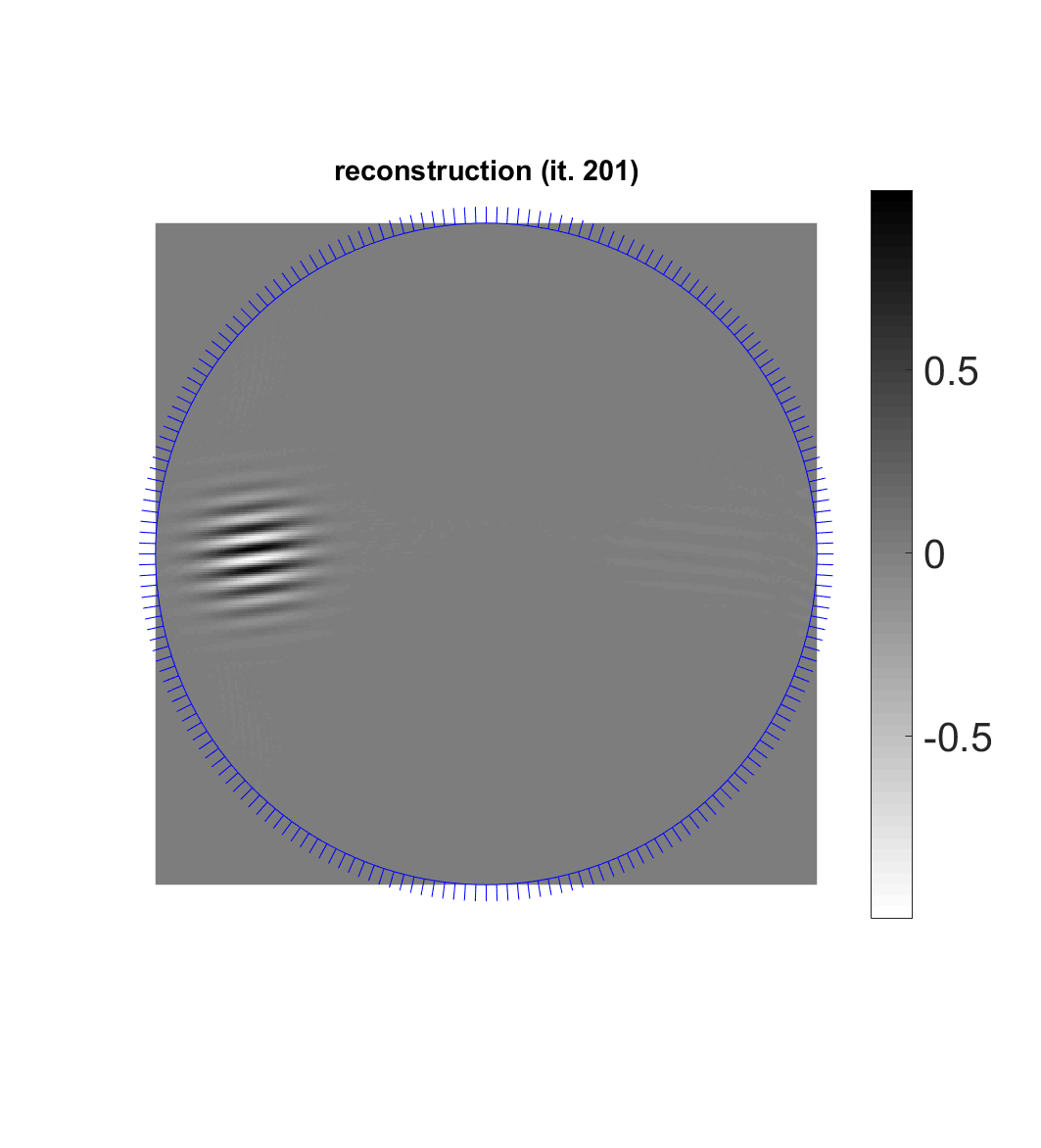}
	\caption{\small Example \ref{ex5}: two conjugate points, positive attenuation. True $f$ is given in Figure~\ref{2pts_coherent_zero} (left). Left to right: attenuation $a$; $f^{(1)}$; $f^{(201)}$.  The $L^\infty$ error is about $3\%$. The $101^{\text{st}}$ iteration is only slightly worse.}
	\label{2pts_coherent_positive}
    \end{figure}
\end{example}

\begin{example}[Figure~\ref{fig:local}: speed $c_3$, local considerations] \label{ex:local}
    We choose $f$ to be an approximate Dirac at $x_0 = (-0.75,0)$, whose conjugate locus consists of two connected components behind each lens of $c_3$, and the attenuation is smooth and equals $2$ inside the dashed circle on Figure~\ref{fig:local} (right). In particular, the attenuation is supported between $f$ and the lower connected component of the conjugate locus of $x_0$. As predicted by the theory, after running Landweber iterations, the lower part of the conjugate locus does not appear as an artifact (because $\det Q\ne 0$ in the microlocal $2\times 2$ systems associated with those pairs of conjugates points), while the upper part does (because $\det Q= 0$ in the microlocal $2\times 2$ system associated with those pairs of conjugate points), see Figure~\ref{fig:local}. This illustrates the (micro-)locality of the concept of stability. 

    \begin{figure}[htpb]
	\centering
	\includegraphics[trim = 30 70 10 90, clip, height=.21\textheight]{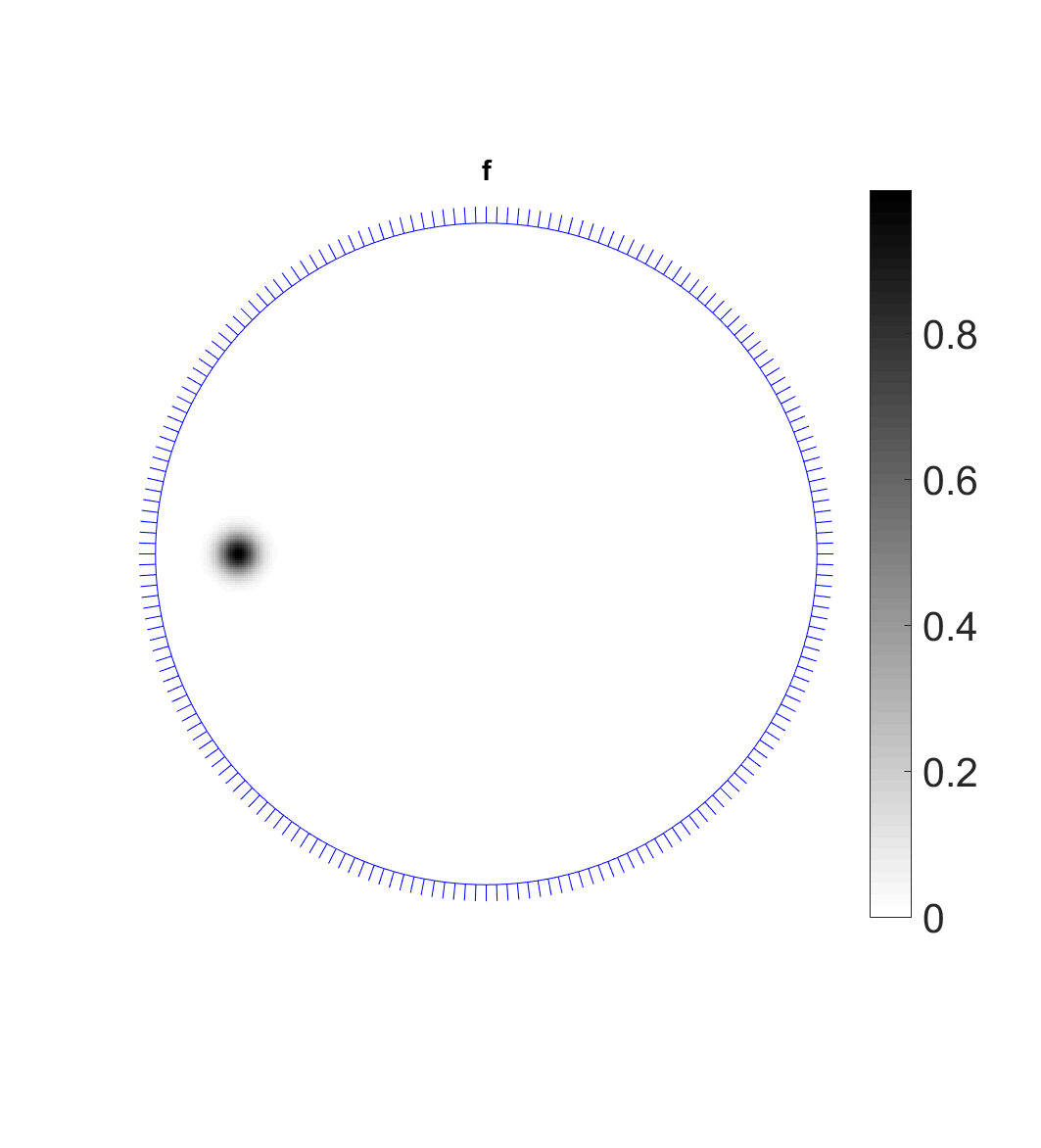}
	\includegraphics[trim = 30 70 10 90, clip, height=.21\textheight]{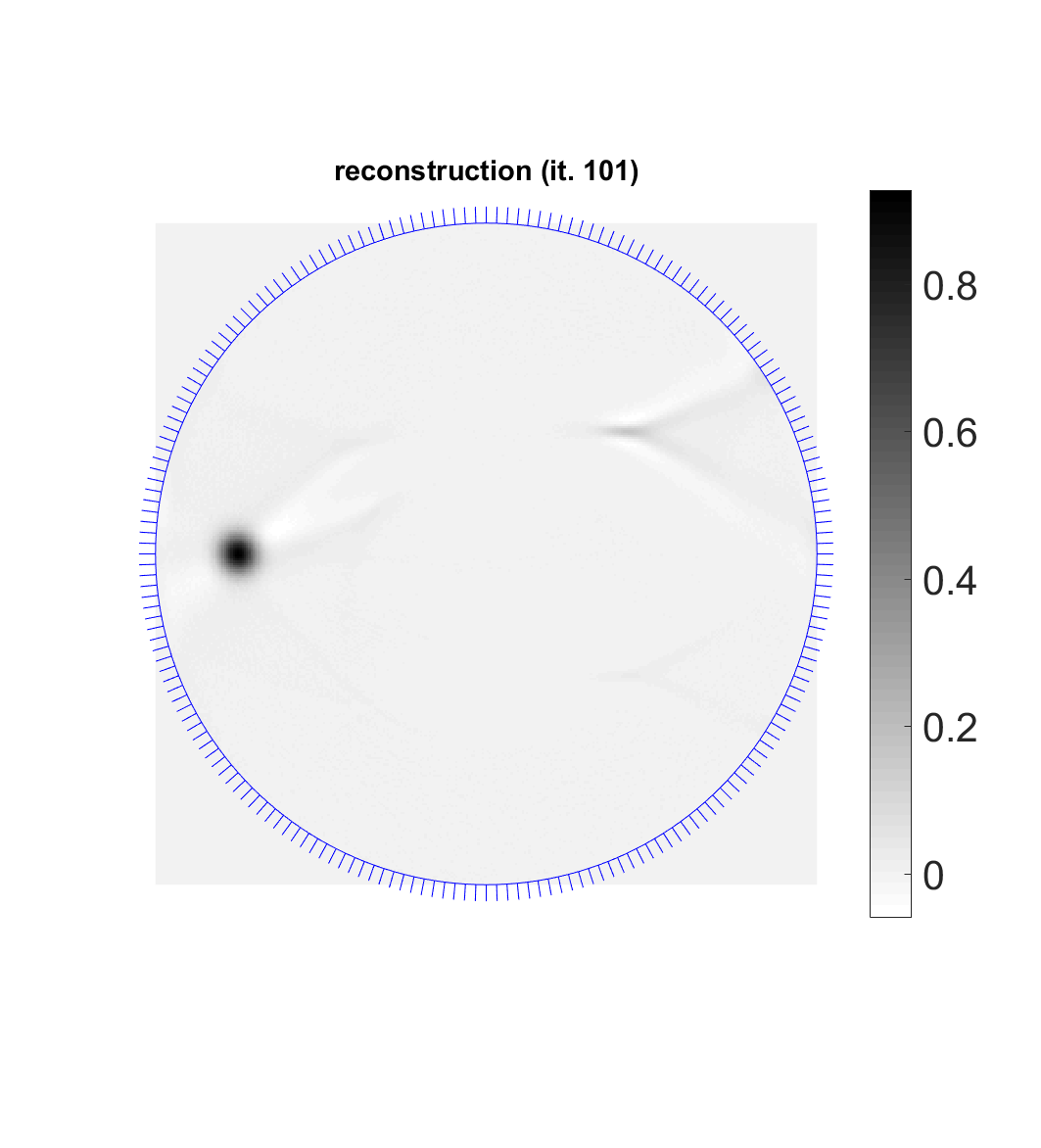}	
	\includegraphics[trim = 30 70 10 90, clip, height=.21\textheight]{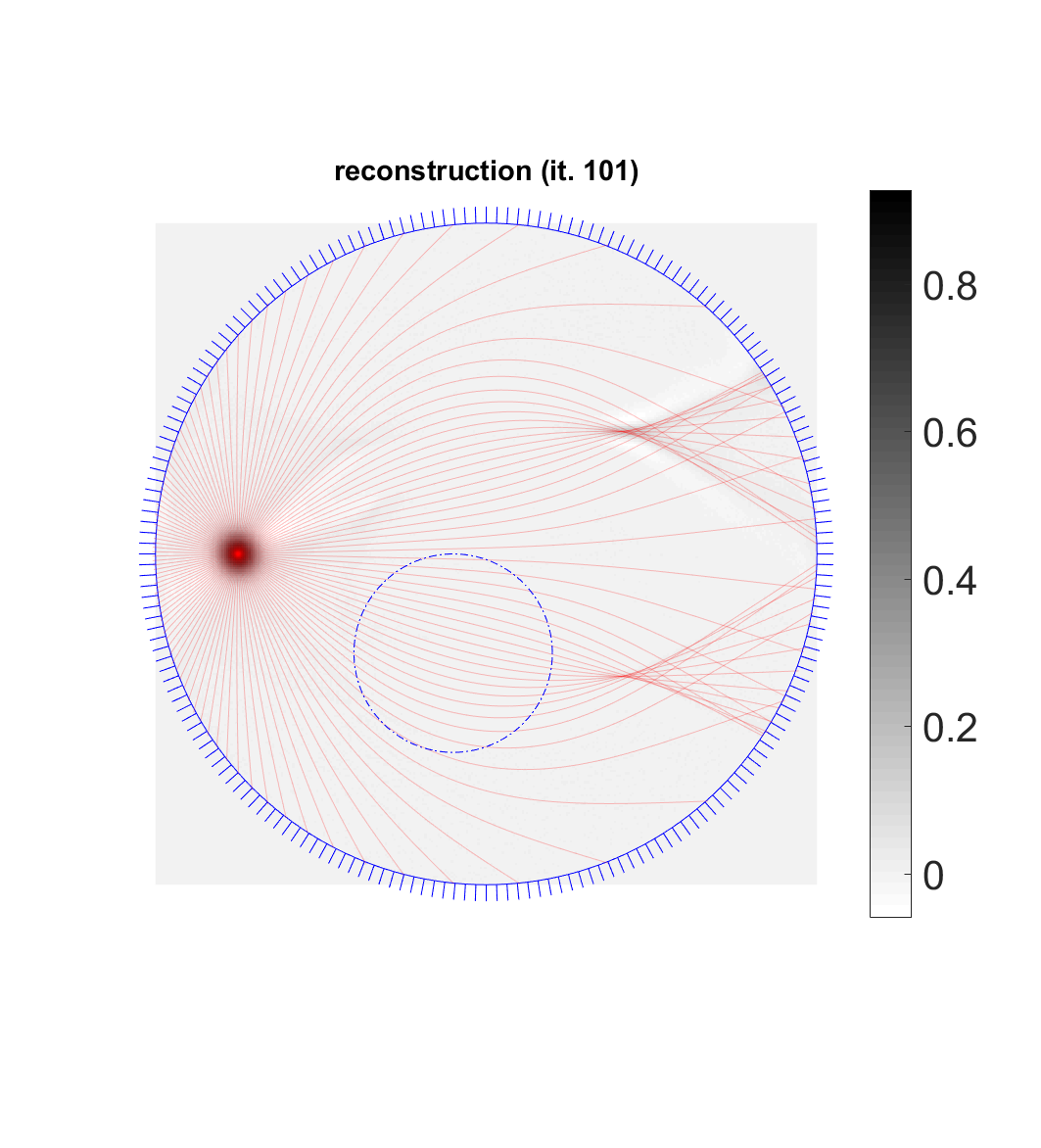}
	\caption{\small Example \ref{ex:local}: two conjugate points per geodesic, local considerations. Left to right: true $f$; $f^{(101)}$ alone; $f^{(101)}$ with geodesics delineating the conjugate locus of $f$, with a dashed circle representing the support of $a$. No artifact appears at the bottom due to the presence of $a>0$ there. }
	\label{fig:local}
    \end{figure}
\end{example}

\begin{example}[Figure~\ref{3pts}: speed $c_2$, a coherent state. Zero vs.\ non-zero attenuation] \label{ex3}
We change\\ the metric to make sure that there are three conjugate points along the ``gutter''. We chose $f$ as a coherent state as in \r{coh} but centered at $(0.05, 0.1)$. 
    The corresponding singularities are not recoverable in line with the analysis in Section~\ref{sec_3pts}. 
    \begin{figure}[h!] 
	\centering 
	\includegraphics[trim = 30 70 10 90, clip, height=.2\textheight]{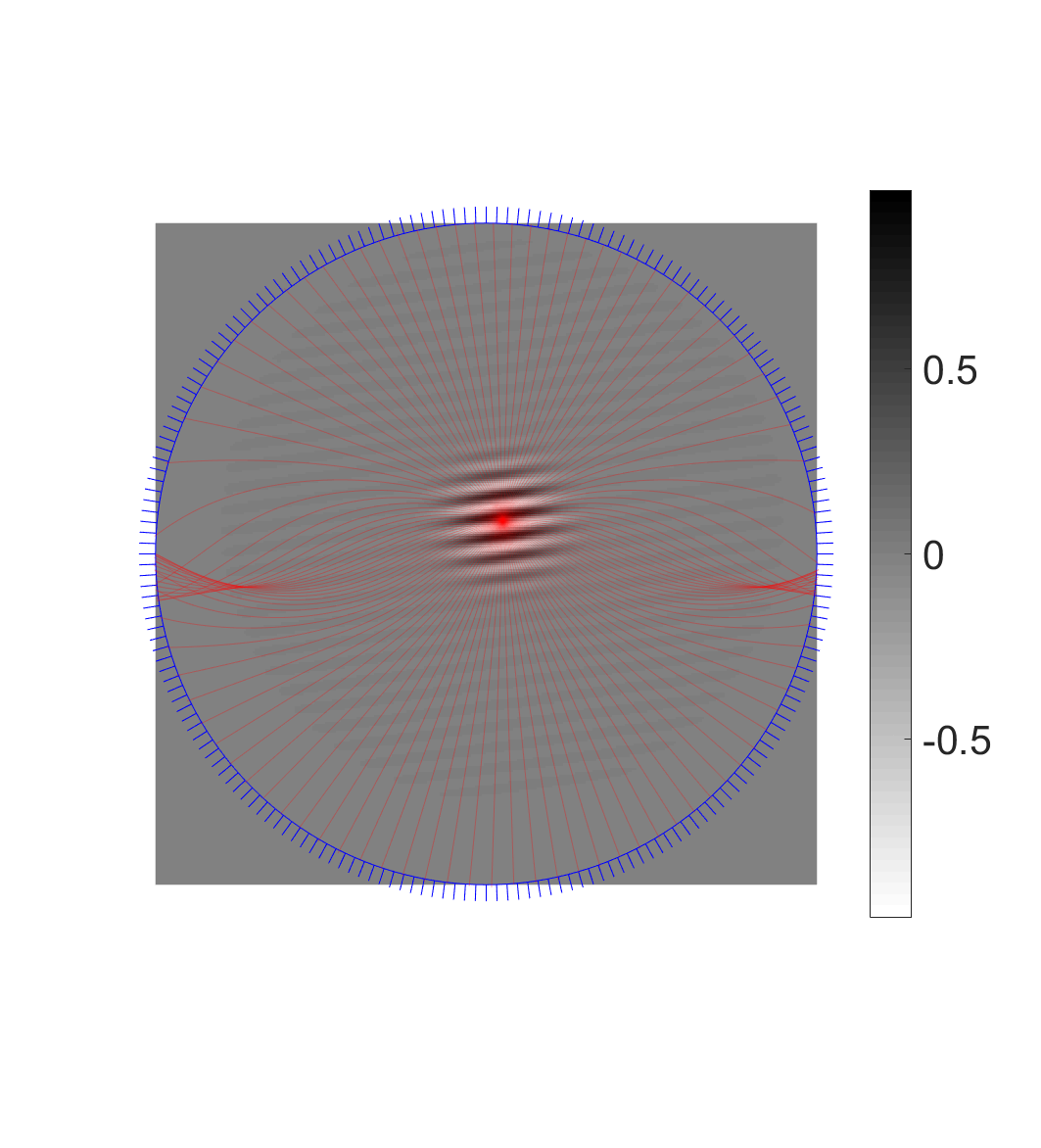} 
	\includegraphics[trim = 30 70 10 90, clip, height=.2\textheight]{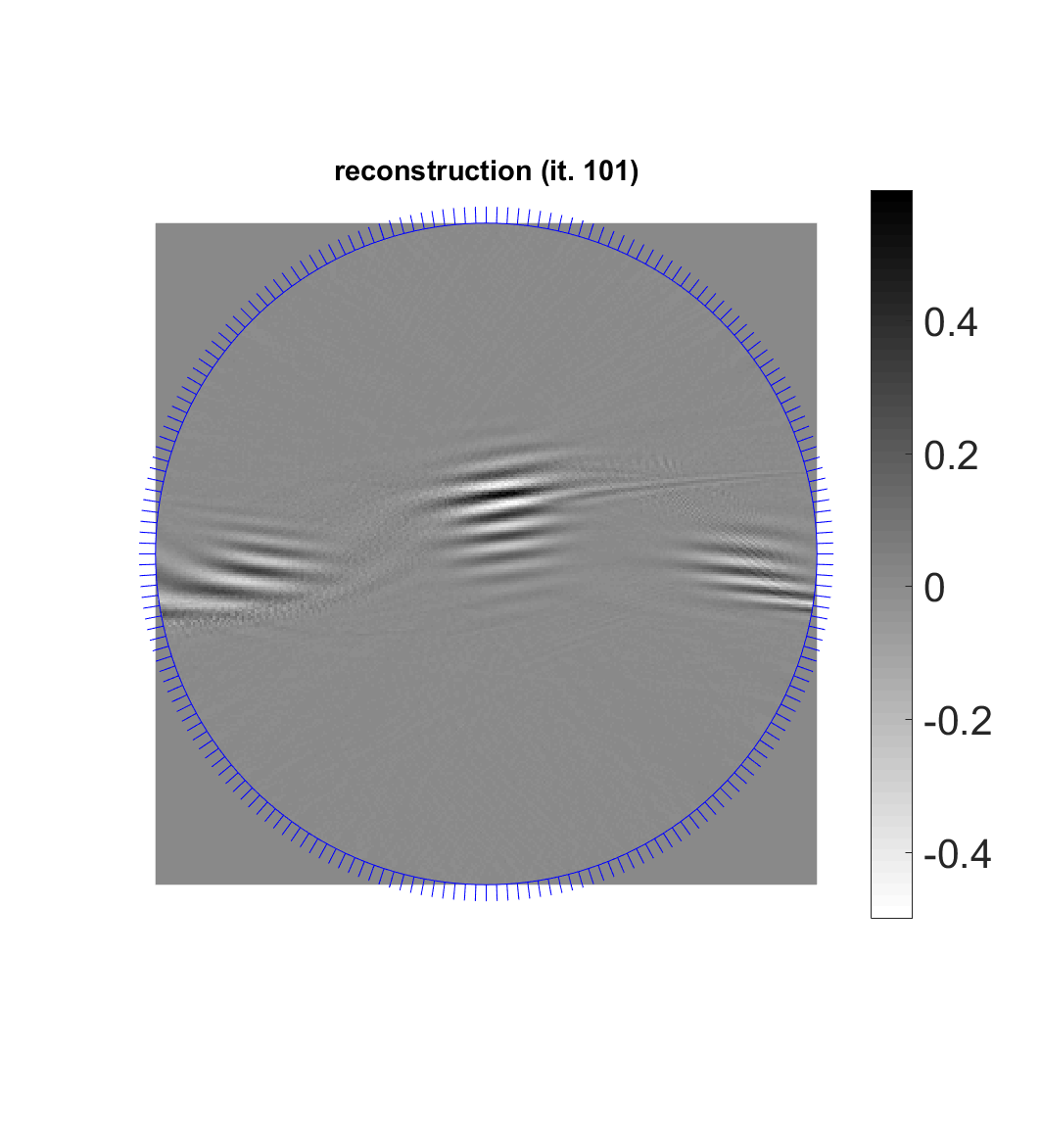}
	\includegraphics[trim = 30 70 10 90, clip, height=.2\textheight]{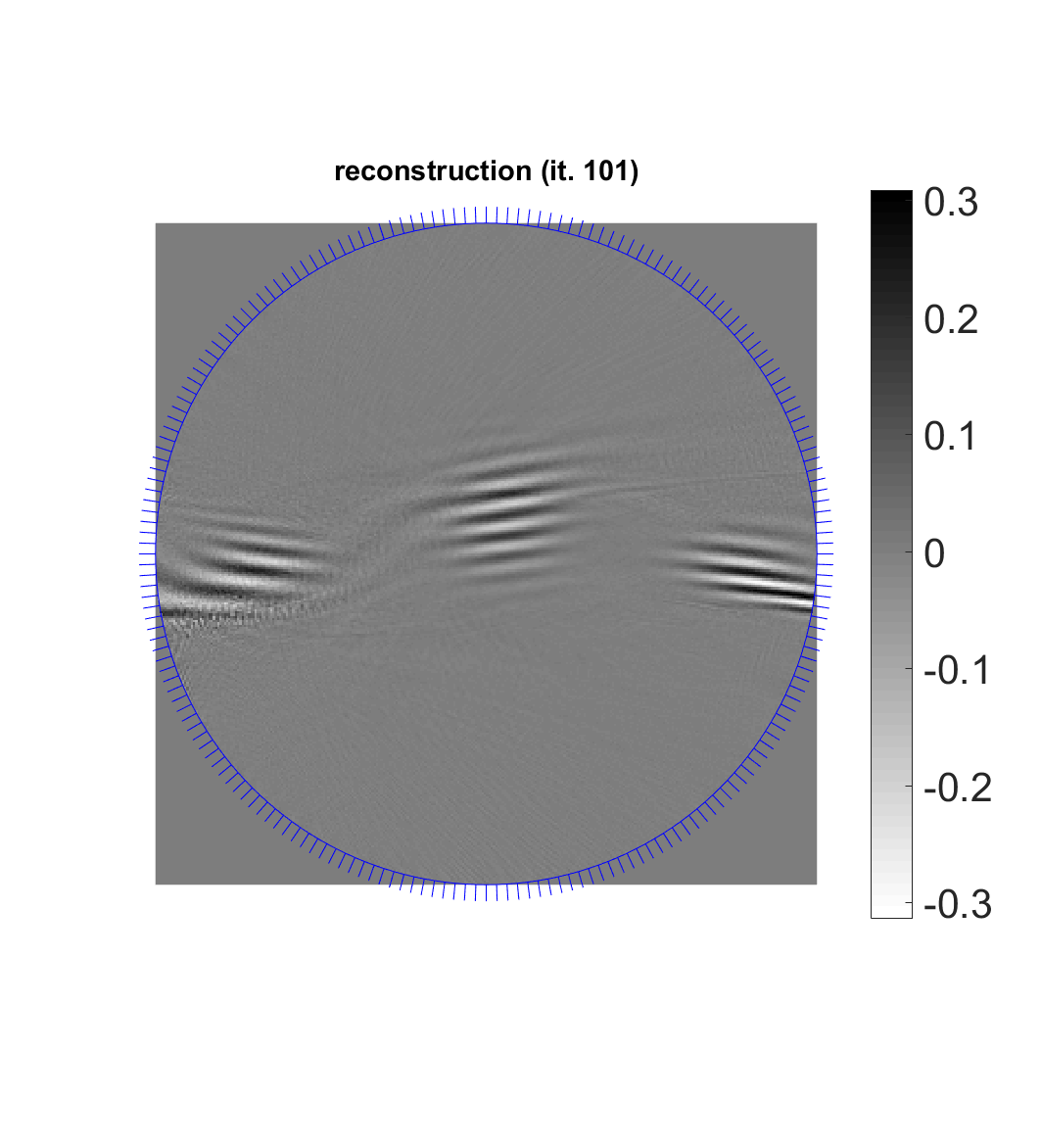}
	\caption{\small Example \ref{ex3}: three conjugate points, zero and positive attenuation. Left to right: the true $f$ with geodesics superimposed; $f^{(101)}$ reconstructed with zero attenuation; $f^{(101)}$ reconstructed with the attenuation displayed on Figure \ref{2pts_coherent_positive}. Artifacts appear in both cases and taking more iterations only increases the noise-like artifacts.}
	\label{3pts}
    \end{figure}
\end{example}

\begin{example}[Figure~\ref{3Dfig}: Three dimensional reconstruction] \label{ex3d} 
    In this example we consider a three dimensional reconstruction giving a numerical illustration of the theoretical discussion in Section~\ref{sec_2.7}. We use the same type of ``gutter metric" $c_2$ as in Example \ref{ex3} in which there are three conjugate points along some geodesics tangent to the direction of the gutter, and the phantom is a coherent state aligned with the gutter. However, in this case the geodesics normal to the gutter do not have conjugate points, and so we have a stable reconstruction.     The geodesics on one plane through the origin are shown as well as a volume rendering of the reconstructed phantom. In this reconstruction we have used the numerical method LSMR \cite{FongLSMR}, applied to a sparse matrix arising from discretisation, rather than the Landweber iteration.
    \begin{figure}[h!]
	\includegraphics[trim = 0 -40 0 22, clip, height=.205\textheight]{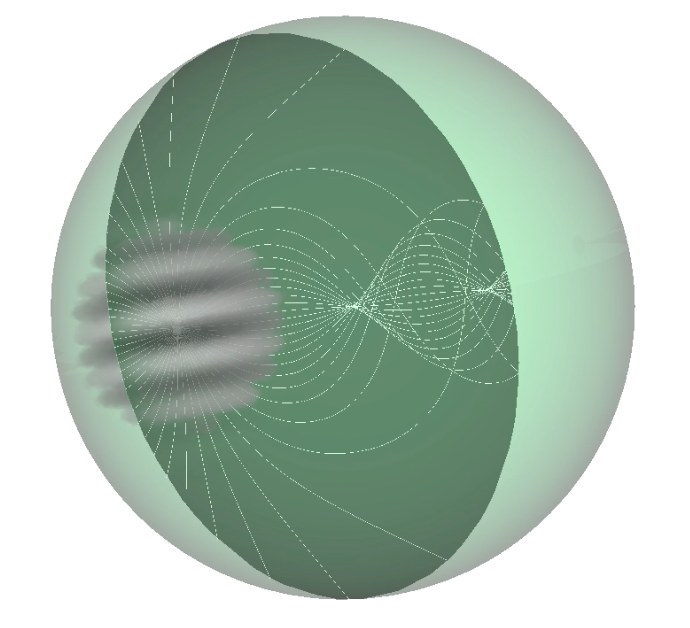} 
	\caption{Example \ref{ex3d}: a three dimensional reconstruction with no added noise illustrating the stable reconstruction. Geodesics on a plane through the origin are shown with conjugate points occurring along the gutter including three conjugate points along some geodesics. This should be contrasted with Example~\ref{ex3}.}
	\label{3Dfig}
    \end{figure}
\end{example}

\begin{example}[Figure~\ref{2pts_coherent_positive2}: noisy data, speed $c_1$, no artifacts, zero attenuation] \label{ex6} \ 
    We  choose $f$ to be a coherent state as in Example \ref{ex4} but we add a Gaussian to make sure that $f\ge0$ . The attenuation is zero.   This metric is the same as in  Example~\ref{ex1} and Example~\ref{ex2} and it has conjugate points. We place the coherent state close to the center. Even though there are conjugate points, the geodesics conormal to the singularities of $f$ do not have such points. Without noise, the recovery is excellent with about $2\%$  error in the $L^\infty$ norm in the $201^{\text{st}}$ iteration.  The $101^{\text{st}}$ one is very similar with a similar error but we present the $201^{\text{st}}$ one to show that there is no divergence tendency even up to $k=201$.
    
    \begin{figure}[h!] 
	\centering 
	\includegraphics[trim = 0 50 0 90, clip, height=.21\textheight]{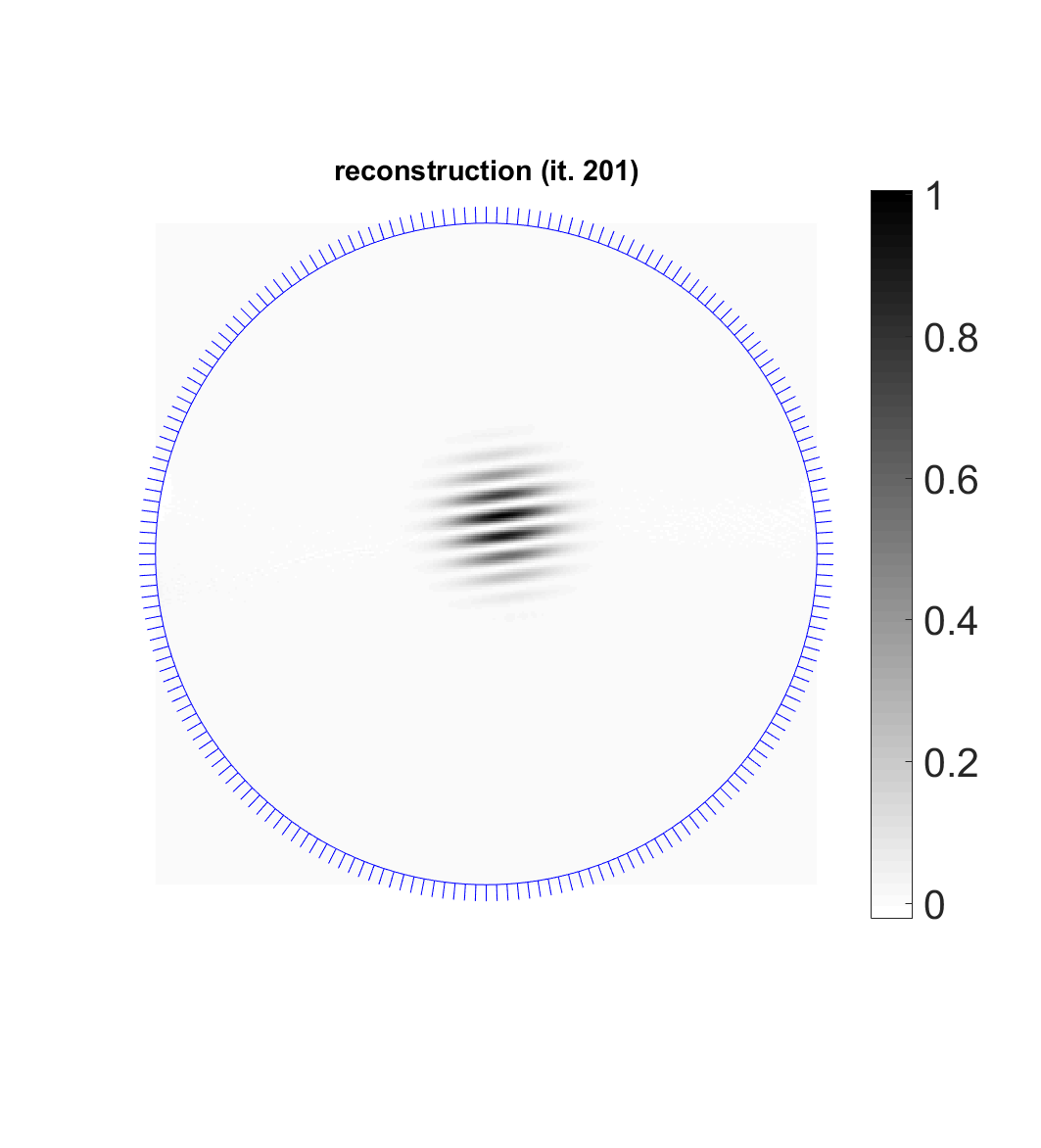}
	\includegraphics[trim = 30 70 10 90, clip, height=.21\textheight]{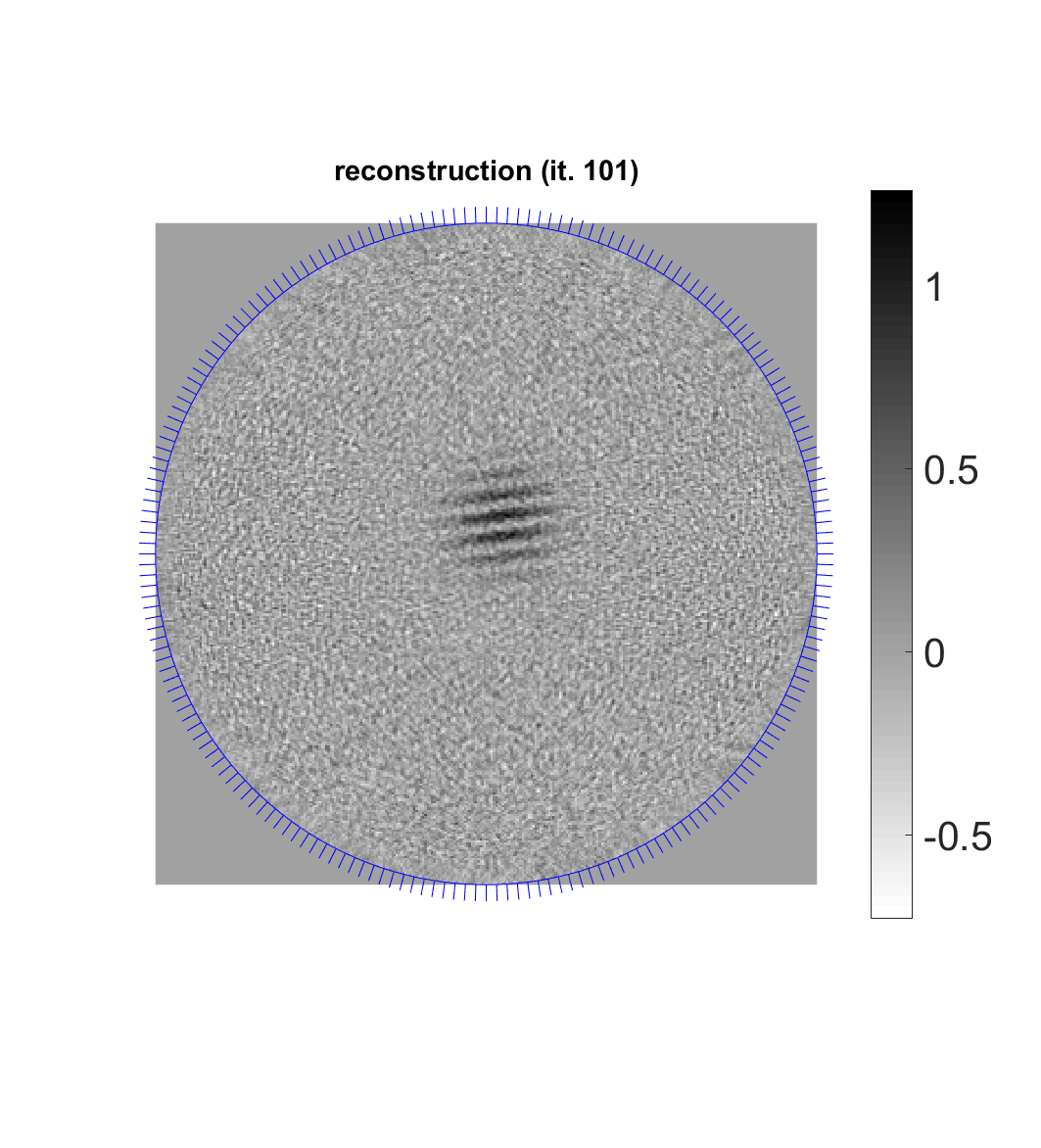}
	\includegraphics[trim = 30 70 10 90, clip, height=.21\textheight]{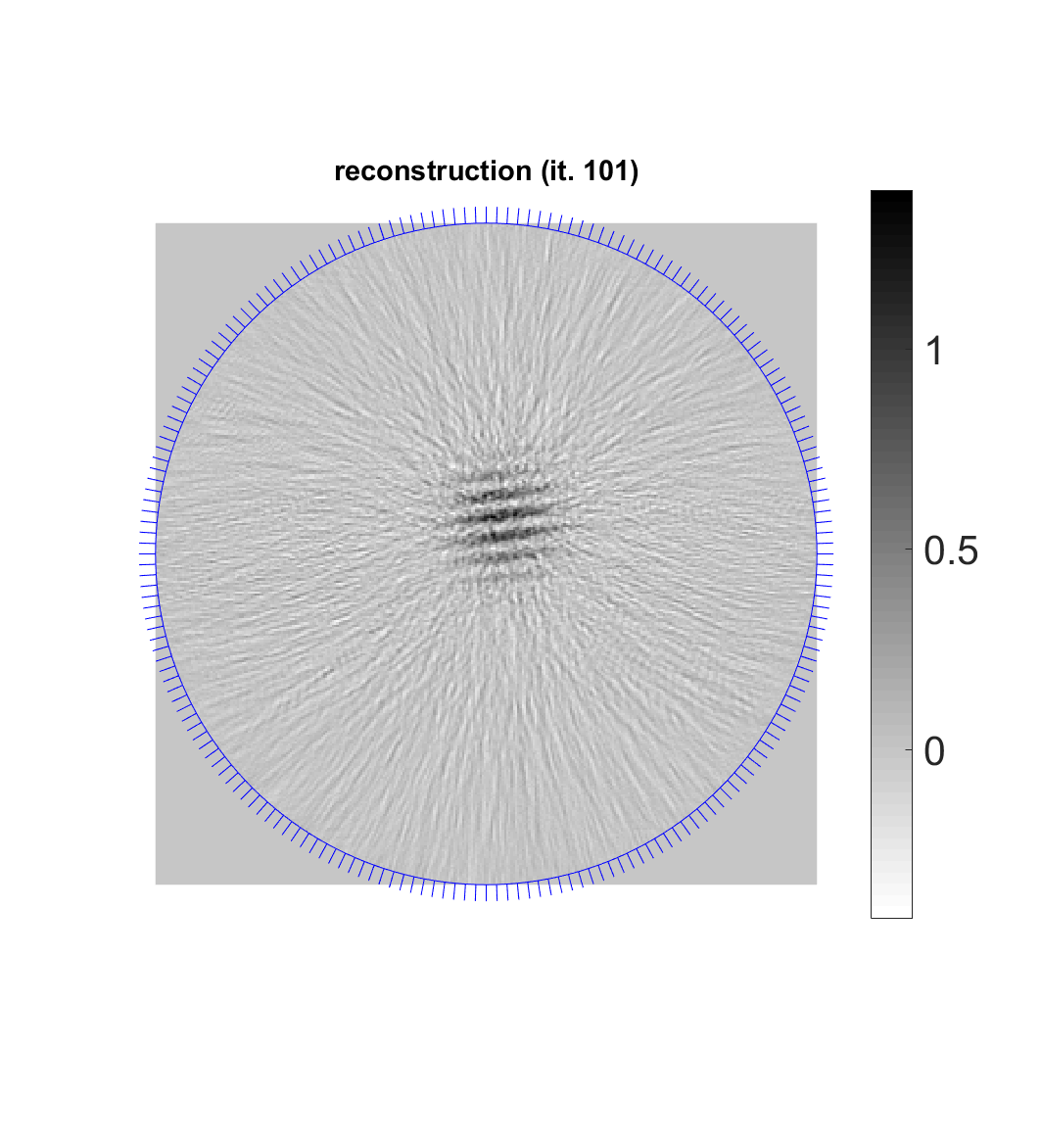}
	\caption{\small Example \ref{ex6}: two conjugate points, zero  attenuation. Left to right: $f^{(201)}$ without noise, visibly identical to the true $f$; $f^{(101)}$ with Gaussian noise added to $Xf$; $f^{(101)}$ reconstructed from $Xf$ modulated by Poisson noise.}
	\label{2pts_coherent_positive2}
    \end{figure}
\end{example}

In the second example, we added Gaussian noise with a standard deviation about $17\%$ of $\|Xf\|_{L^\infty}$. 
In the third case, we modulated $Xf$ by Poisson noise. The computations are done on a $300\times 300$ grid and $Xf$ is a $300\times 600$ matrix in fan-beam coordinates. In those coordinates, the range of $Xf$ is approximately $[0,0.32]$.  We scaled $Xf$ to take the range to approximately $[0,10]$, randomized each entry by Poisson noise with mean equal to its value, and then rescaled in back. Note that this induces noise with standard deviation $\sqrt{10}$ before the rescaling at the highest values of $Xf$, and noise to signal level is  $1/\sqrt{10}\approx 0.32$ there, independent of the scaling. 

This example reveals several interesting features. First, without noise, the reconstruction is close to perfect despite the presence of conjugate points! This is consistent with our analysis. The singularities of $f$ do not belong to the microlocal kernel of $X$, compare with Example~\ref{ex4}. Therefore, they can be stably reconstructed and they would not create artifacts. Next, we do not get artifacts at conjugate points (an element of the microlocal kernel)  despite the fact that an arbitrary inversion would add such an element to the reconstruction. The reason is that the Lanwdweber iterations could only add such an element created by $f$, see, e.g., \r{L1} and the discussion in section~\ref{sec_3.3} in general. Another point of view is that the spectral representation $\tilde f$ of $f$ has low density near $\lambda=0$ because $f$ is separated from the microlocal kernel. 
So we have an example of a unstable problem for which the Landweber iterations work well. 

Next, in the presence of noise, convergence is not guaranteed (and generically not true). The best iteration is around the $50^{\text{th}}$ one in both cases with noise, with no visible error (relative to the noise) in the way the singularity is recovered; after it, the noise levels increase.  The inversion is still good considering the noise and the existence of conjugate points does not appear on the reconstruction as conjugate locus artifacts. They are not visible even under a close inspection of the error (not shown here) probably because they are dominated by the noise. We expect such artifacts to show up for a much higher number of iterations but for $k=101$ and $\gamma$ we choose, the regularizing effect takes over. 

\bibliographystyle{abbrv}

\end{document}